\documentclass[10pt,reqno,a4paper]{amsart}

\usepackage{amsmath,amssymb,amsthm,mathtools,dsfont,bm,bbm,mathrsfs,color} 

\newtheorem{lem}{Lemma}[section]
\newtheorem{prop}{Proposition}[section]
\newtheorem{cor}{Corollary}[section]

\newtheorem{thm}{Theorem}[section]
\newtheorem*{thmn}{Theorem}

\theoremstyle{definition}

\theoremstyle{remark}

\theoremstyle{remark}

\newtheorem*{remarks*}{Remarks}
\newtheorem*{remark*}{Remark}
\numberwithin{equation}{section}

\newcommand{\AP}{\mathrm{AP}}
\newcommand{\BC}{\mathrm{BC}}

\newcommand{\C}{\mathbb{C}}
\newcommand{\N}{\mathbb{N}}
\newcommand{\Z}{\mathbb{Z}}
\newcommand{\R}{\mathbb{R}}
\newcommand{\T}{\mathbb{T}}
\newcommand{\D}{\mathbb{D}}

\newcommand{\Ss}{\mathbb{S}}

\newcommand{\Mean}{\mathscr{M}}
\newcommand{\ov}{\overline}
\newcommand{\ran}{\mathrm{Ran}}
\renewcommand{\ker}{\mathrm{Ker}}
\newcommand{\dom}{\mathrm{Dom}}
\newcommand{\rank}{\mathrm{Rank}}

\newcommand{\pt}{\partial}

\newcommand{\Cd}{\mathbb{C}^{d \times d}}

\renewcommand{\sp}{\sigma_{\mathrm{p}}}

\newcommand{\be}{\begin{equation}}
\newcommand{\ee}{\end{equation}}

\newcommand{\eps}{\varepsilon}

\newcommand{\Xu}{X_{\mathrm{u}}}
\newcommand{\Xs}{X_{\mathrm{s}}}

\newcommand{\weakto}{\rightharpoonup}
\newcommand{\ii}{i}
\newcommand{\eu}{\mathrm{e}}

\renewcommand{\rho}{\varrho}

\newcommand{\uu}{\mathbf{u}}
\newcommand{\ff}{\mathbf{f}}

\newcommand{\Vb}{\mathbf{V}}

\newcommand{\Ub}{\mathbf{U}}
\newcommand{\Fb}{\mathbf{F}}
\newcommand{\Eb}{\mathbf{E}}

\newcommand{\Gb}{\mathbf{G}}

\newcommand{\id}{I}
\newcommand{\Tr}{\mathrm{Tr}}
\newcommand{\tr}{\mathrm{tr}}

\newcommand{\Gr}{\mathsf{Gr}}
\newcommand{\Om}{\Omega}

\newcommand{\sd}{\sigma_{\mathrm{d}}}
\newcommand{\se}{\sigma_{\mathrm{e}}}

\renewcommand{\phi}{\varphi}

\newcommand{\Ks}{\mathscr{K}}

\newcommand{\Es}{\mathscr{E}}
\newcommand{\Hs}{\mathscr{H}}
\newcommand{\Us}{\mathscr{U}}
\newcommand{\Ls}{\mathscr{L}}

\newcommand{\Vs}{\mathscr{V}}

\newcommand{\Qb}{\mathbf{Q}}

\newcommand{\comment}[1]{}

\catcode`@=11
\def\section{\@startsection{section}{1}%
  \z@{1.5\linespacing\@plus\linespacing}{.5\linespacing}%
  {\normalfont\bfseries\large\centering}}
\catcode`@=12

\makeatletter
\def\@cite#1#2{[\textbf{#1\if@tempswa , #2\fi}]}
\def\@biblabel#1{[\textbf{#1}]}
\makeatother

\begin{document}
\title[HWM on $\T$: GWP in $H^{1/2}$ and almost periodicity]{The half-wave maps equation on $\T$: \\ Global well-posedness in $H^{1/2}$ \\ and Almost Periodicity}

\author{Patrick G\'erard}
\address{P. G\'erard,  Laboratoire de Math\'ematiques d'Orsay, CNRS, Universit\'e Paris-Saclay, 91405 Orsay, France.}%
\email{patrick.gerard@universite-paris-saclay.fr}

\author{Enno Lenzmann}
\address{E. Lenzmann, University of Basel, Department of Mathematics and Computer Science, Spiegelgasse 1, CH-4051 Basel, Switzerland.}%
\email{enno.lenzmann@unibas.ch}

\maketitle

\begin{abstract}
We consider the half-wave maps equation
\[
\partial_t \uu = \uu \times |D| \uu
\]
for $\uu : \R \times \T \to \Ss^2$, where $\T=\R/2 \pi \Z$ is the one-dimensional torus and $\Ss^2 \subset \R^3$ denotes the unit sphere. By extension from rational initial data, we construct a unique and continuous flow map for data in the critical energy space $H^{1/2}(\T; \Ss^2)$. Moreover, we show almost periodicity in time of these solutions. For the dense subset of rational initial data, we establish quasi-periodicity in time and a-priori bounds on $\| \uu(t) \|_{H^s(\T)}$ for any $s >0$.

Our analysis relies crucially on an explicit formula arising from the Lax pair structure acting on a Hardy space of vector-valued holomorphic functions on the unit disk. As a central ingredient, we develop a general {\em stability principle} for explicit formulae associated with completely integrable PDEs possessing a Lax pair structure on Hardy spaces, including the Benjamin--Ono equation, Calogero--Sutherland DNLS, and the half-wave-maps equation posed on $\T$.

Our results extend to the matrix-valued half-wave maps equation 
$$
\pt_t \Ub = -\frac{i}{2} [ \Ub, |D| \Ub ]
$$
with target manifold given by the complex Grassmannians $\Gr_k(\C^d)$, thereby generalizing the special case $\Ss^2 \cong \mathbb{CP}^1 \cong \Gr_1(\C^2)$. In a companion work, we prove global well-posedness for the half-wave maps equation posed on $\R$ in the scaling-critical energy space $\dot{H}^{1/2}$, by establishing a stability principle for explicit formulae on Hardy spaces in the complex half-plane $\C_+$. 
\end{abstract}

\tableofcontents

\section{Introduction} \label{sec:intro}

\subsection{Formulation of the problem} As a starting point of this paper, we consider the half-wave maps equation (HWM) posed on the one-dimensional torus $\T = \R/2 \pi \Z$ with target $\Ss^2$. The corresponding evolution equation can be written as
\be \label{eq:HWM_S2} \tag{HWM$_{\Ss^2}$}
\boxed{\pt_t \uu = \uu \times |D| \uu}
\ee
for the map $\uu : \R \times \T \to \Ss^2$. Here $\Ss^2$ is the standard unit two-sphere embedded in $\R^3$ and $\times$ stands for the vector product in $\R^3$. As usual, the operator $|D|$ denotes the first-order fractional derivative for functions on $\T$, i.e., 
$$
|D| \ff(\theta) = \sum_{n =-\infty}^{\infty} |n| \, \widehat{\ff}_n \eu^{\ii n \theta} \quad \mbox{for $\theta \in \T$}.
$$ 
We refer to \cite{ZhSt-15, LeSc-18, Ma-22, BeKlLa-20, LeSo-24, GeLe-25} for recent works on \eqref{eq:HWM} and its related version posed on $\R$.  

\medskip
The half-wave maps equation was introduced independently in \cite{LeSc-18, ZhSt-15}, motivated by the theory of half-harmonic maps and by connections with integrable Calogero--Moser spin systems. It is a Hamiltonian equation whose energy functional is
$$
E[\uu]  = \frac{1}{2} \int_\T \uu \cdot |D| \uu = \frac{1}{2} \sum_{n=-\infty}^\infty |n| |\widehat{\uu}_n|^2 =  \frac{1}{2} \| \uu \|_{\dot{H}^{1/2}}^2.
$$
The natural energy space is therefore
$$
H^{1/2}(\T; \Ss^2) = \{ \uu \in H^{1/2}(\T; \R^3) : \mbox{$\uu(\theta) \in \Ss^2$  a.e.} \}.
$$
Geometrically, this is a loop space, and \eqref{eq:HWM} may be viewed as a Hamiltonian flow on the infinite-dimensional symplectic space $H^{1/2}$-loops into $\Ss^2$. The energy is  {\em conformally invariant}. Indeed, writing
$$
E[\uu] = \frac{1}{4 \pi} \iint_{\T \times \T} \frac{|\uu(\theta)-\uu(\phi)|^2}{2 \sin^2 (\frac{\theta-\phi}{2} ) } d \theta \, d \phi = \frac{1}{4 \pi} \int_{\D} |\nabla \uu^e(x,y)|^2 \, dx \,dy,
$$
where $\uu^e : \D \to \R^3$ denotes the harmonic extension of $\uu \in H^{1/2}(\T; \Ss^2)$ to the unit disk $\D$, one obtains 
$$
E[\uu \circ \phi] = E[\uu]
$$
for every conformal automorphism $\phi : \D \to \D$. Critical points of $E$ are referred to as {\em half-harmonic maps}, whose regularity theory was developed in the seminal work of Da Lio--Rivi\`ere \cite{LiRi-11}; see also \cite{MiSi-15,LeSc-18, MaSc-18}.

Local well-posedness of \eqref{eq:HWM} for $H^s$-data with $s > \frac{3}{2}$ follows by an iterative scheme for hyperbolic systems. However, despite the complete integrability of \eqref{eq:HWM}, the global well-posedness of the Cauchy problem has remained open, even for smooth data sufficiently close to a constant state. The main obstacles are:
\begin{itemize}
\item The equation is a {\em quasi-linear} system. 
\item The operator $|D|$ provides no dispersion in one space dimension.
\item The Lax structure for \eqref{eq:HWM} does not control Sobolev norms above $H^{1/2}$.
\end{itemize}

One of the principal contributions of this work is to bridge this gap by proving global well-posedness for \eqref{eq:HWM} in the critical energy space $H^{1/2}$. The construction proceeds by first proving global well-posedness for all rational initial data, where we adapt our recent approach in \cite{GeLe-25} from $\R$ to $\T$. Now, the density of rational data in the energy space $H^{1/2}$ combined with an explicit formula for \eqref{eq:HWM} enables us to obtain a unique extension yielding weak solutions that could exhibit loss of energy. However, to construct a decent flow map for $H^{1/2}$-data, we need to prove its strong continuity, which is equivalent to proving energy conservation.  Here, a key ingredient will be a newly found {\em stability principle} for explicit formulae associated with completely integrable PDEs on Hardy spaces, such as the Benjamin--Ono equation (BO), the Calogero--Sutherland DNLS (CS-DNLS), the cubic Szeg\H{o} and (HWM) in \cite{Ge-23, Ge-24, GeLe-25, GePu-23, Ba-24, KiLaVi-24, KiLaVi-25, GaGeMi-26, Xi-25, Ma-25}. 

With regard to the general context, we emphasize the fact that the analysis of (HWM) prevents us from adapting known tools developed to prove global well-posedness results for other dispersive geometric PDEs such as the Schr\"odinger maps or wave maps equations; see e.g.~\cite{Ta-06,KotaVi-14} and references therein. We refer also to \cite{KrSi-18, KiKr-21, Li-23} for small data global existence for \eqref{eq:HWM} in the non-integrable case posed on $\R^N$ with $N \geq 3$, where dispersive estimates can be used that are not available for our setting here. 


\subsection{Matrix-valued generalization of (HWM)} 
Before we state our main results, it will be useful to introduce the following generalization of the half-wave maps equation beyond the target $\Ss^2$. By following \cite{GeLe-18}, we notice that \eqref{eq:HWM_S2} can be rephrased in matrix-commutator form given by
\be \label{eq:HWM} \tag{HWM}
\boxed{\pt_t \Ub = -\frac{i}{2} [\Ub, |D| \Ub]}
\ee
with the matrix-valued map 
$$
\Ub = \uu \cdot \bm{\sigma} = \sum_{k=1}^3 u_k \sigma_k = \left ( \begin{array}{cc} u_3 & u_1 - i u_2 \\ u_1 + i u_2 & -u_3 \end{array} \right ),
$$
where $\bm{\sigma} = (\sigma_1, \sigma_2, \sigma_3)$ denote the standard Pauli matrices. It is straightforward (see \cite{GeLe-18}) to verify that the geometric constraint that $\uu(t,\cdot)$ takes values in $\Ss^2$ is equivalent to the algebraic conditions 
$$
\Ub(t,\cdot) = \Ub(t, \cdot)^*, \quad \Ub(t,\cdot)^2 = \mathds{1}_2, \quad \Tr (\Ub(t,\cdot)) = 0.
$$ 
As recently proposed in \cite{GeLe-25}, the matrix-valued formulation of (HWM) leads to a natural generalization by considering maps $\Ub$ valued in the {\em complex Grassmannians}, which we denote by $\Gr_k(\C^d)$ and identify with the set of matrices
$$
\Gr_k(\C^d) := \left \{ U \in \C^{d \times d} : U^* = U, \; U^2 = \mathds{1}_d, \; \Tr (U) = d-2k \right \},
$$
where $d \geq 2$ and $0 \leq k \leq d$ are given integers.\footnote{Note the trivial cases $\Gr_0(\C^d) = \{ \mathds{1}_d \}$ and $\Gr_d(\C^d) = \{ -\mathds{1}_d \}$. To streamline the presentation, we will mostly consider  the non-trivial cases $\Gr_k(\C^d)$ with $1 \leq k \leq d-1$.} For any $U \in \Gr_k(\C^d)$, we note that $P= \frac{1}{2}(\mathds{1}_d- U)$ is an orthogonal projection $P=P^*=P^2$ with $\rank(P) = \Tr(P)=k$. Thus, we can canonically identify elements in $\Gr_k(\C^d)$ with the $k$-dimensional subspaces in $\C^d$ in accordance with the geometric definition of the Grassmannian $\Gr_k(\C^d)$. Furthermore, we recall that $\Gr_k(\C^d)$ is a compact K\"ahler manifold of complex dimension $k(d-k)$ and the Grassmannians $\Gr_1(\C^d)$ clearly correspond to the projective spaces $\mathbb{CP}^{d-1}$ which are particularly interesting from the physical point of view.

For the rest of this paper, we will thus study (HWM) with the targets $\Gr_k(\C^d)$, where we remind the reader that all results shown below also apply to \eqref{eq:HWM_S2} by using that $\Ub = \uu \cdot \bm{\sigma} \in \Gr_1(\C^2)\cong \mathbb{CP}^1$ with $\uu = (u_1,u_2,u_3) \in \Ss^2$ if and only if $u_k = \frac{1}{2} \Tr(\Ub \sigma_k)$ for $k=1,2,3$.

\subsection{Lax pair structure}
For \eqref{eq:HWM} with target $\Gr_k(\C^d)$, the energy functional is readily found to be
\be \label{def:energy_HWM}
E[\Ub] = \frac{1}{2} \int_\T \Tr(\Ub (|D|\Ub)) \, d\theta = \frac{1}{2} \sum_{n =-\infty}^\infty |n| |\widehat{\Ub}_n|^2 = \frac{1}{2} \| \Ub \|_{\dot{H}^{1/2}}^2
\ee
with  the corresponding energy space
$$
H^{1/2}(\T; \Gr_k(\C^d)) = \{ \Ub \in H^{1/2}(\T; \C^{d \times d}) : \mbox{$\Ub(\theta) \in \Gr_k(\C^d)$ for a.e.~$\theta \in \T$} \}.
$$
As before, we notice that the energy $E[\Ub]$ is conformally invariant and hence the energy space $H^{1/2}(\T; \Gr_k(\C^d))$ is critical for (HWM) with respect to this invariance of the problem. From the analysis of half-harmonic maps, we infer that \eqref{eq:HWM} has non-trivial {\em stationary solutions} and more generally {\em traveling solitary waves}, see the appendix for details.

Let us now briefly discuss the Lax pair structure behind the half-wave maps equation. To this end, we assume that $\Ub \in C(I; H^s)$ for some $s > \frac 3 2$ solves \eqref{eq:HWM} on some time interval $I$ with $0 \in I$. For any $t \in I$, we consider the self-adjoint Toeplitz operator
\be \label{def:Toep}
T_\Ub(t) : L^2_+ \to L^2_+, \quad \Fb \mapsto T_{\Ub(t)} \Fb = \Pi (\Ub(t) \Fb) 
\ee
defined in the matrix-valued Hardy space $L^2_+=L^2_+(\T; \C^{d \times d})$, where $\Pi$ denotes the Cauchy--Szeg\H{o} projection onto $L^2_+$. By following \cite{GeLe-18, GeLe-25}, we exploit the algebraic identities for the matrix-valued map $\Ub$ together with product identities for Toeplitz operators to conclude that the following {\em Lax equation} holds:
\be \label{eq:Lax_intro}
\frac{d}{dt} T_{\Ub(t)} = [B_{\Ub(t)}, T_{\Ub(t)}] \quad \mbox{for $t \in I$}.
\ee
Here the unbounded operator 
$$
B_{\Ub(t)} = -\frac{i}{2} ( T_{\Ub(t)} D + D  T_{\Ub(t)}) + \frac{i}{2} T_{|D| \Ub(t)}
$$
is essentially skew-adjoint with $\dom(B_{\Ub(t)}) = \{ \Fb \in L^2_+ : D \Fb \in L^2_+ \}$ and $D = -i \pt_\theta$. As a technical aside, we remark that $B_{\Ub(t)}$ is neither bounded from below nor above, i.e., its spectral properties resemble that of a Dirac-type operator and, moreover, its (essential) skew-adjointness cannot be inferred from the Kato--Rellich theorem due to $\| T_\Ub \| = 1$. We refer to Appendix \ref{app:lwp} for details.

By \eqref{eq:Lax_intro}, we infer the Lax evolution of $T_{\Ub(t)}$ by the unitary equivalence
\be \label{eq:T_unitary}
T_{\Ub(t)} = \Us(t) T_{\Ub_0} \Us(t)^*
\ee
with $\Us(t)$ being the unitary map that solves the operator differential equation $\frac{d}{dt} \Us(t) = B_{\Ub(t)} \Us(t)$ for $t \in I$ and $\Us(0) = \id$. 

Evidently, the spectrum of $T_{\Ub(t)}$ is preserved for $t \in I$. To extract a useful set of conserved quantities, we observe that the algebraic property $\Ub(t,\cdot)^2= \mathds{1}_d$ implies 
$$
T_{\Ub(t)}^2 = \id - K_{\Ub(t)} \quad \mbox{with} \quad K_{\Ub(t)}= H^*_{\Ub(t)} H_{\Ub(t)},
$$
with the Hankel operator $H_{\Ub(t)}\Fb = \Pi_-(\Ub(t) \Fb)$ for $\Fb \in L^2_+$ with the projection $\Pi_- = \id -\Pi$ onto $L^2_- = L^2 \ominus L^2_+$. Now, a direct calculation in Section \ref{sec:spec} below yields
$$
E[\Ub(t)] = \frac{1}{2} \| \Ub(t) \|_{\dot{H}^{1/2}}^2 = \Tr(K_{\Ub(t)}).
$$
Thus the operator $K_{\Ub(t)}$ is trace-class if and only if $\Ub(t) \in H^{1/2}$, and from \eqref{eq:T_unitary} we obtain the unitary equivalence
\be \label{eq:K_unitary}
K_{\Ub(t)} = \Us(t) K_{\Ub_0} \Us(t)^*,
\ee
which reproves conservation of energy $E[\Ub(t)] = E[\Ub_0]$ for $t \in I$, provided that $\Ub(t)$ is a sufficiently smooth solution. In fact, from \eqref{eq:K_unitary} we can infer that sufficiently smooth solutions of \eqref{eq:HWM} exhibit the following family of conserved quantities given by
$$
I_p[\Ub] = \Tr(|K_\Ub|^{p/2}) = \Tr( |H_\Ub|^p)  \quad \mbox{for all $p \in [1, \infty)$}
$$
corresponding to the  Schatten-$p$-norms $\| H_\Ub \|_{\mathfrak{S}_p} = \Tr(|H_\Ub|^p)^{1/p}$. By invoking Peller's theorem (see e.g.~\cite{Pe-03}), we obtain the a-priori bounds
\be \label{ineq:Besov}
\sup_{t \in I} \| \Ub(t) \|_{B^{1/p}_{p,p}} \lesssim \| \Ub(0) \|_{B^{1/p}_{p,p}} \quad \mbox{for all $p \in [1,\infty)$}
\ee 
in the class of Besov spaces $B^{1/p}_{p,p}(\T; \C^{d \times d})$, where we recall that $B^{1/2}_{2,2} = H^{1/2}$. However, these bounds are still off by more than one derivative needed to deduce global well-posedness for solutions in $H^s$ with $s > \frac 3 2$. Thus, at the moment, it is far from clear how to obtain a  local well-posedness theory for (HWM) such that the a-priori bounds \eqref{ineq:Besov} can be used to obtain global regularity.\footnote{A slighlty refined analysis shows that an a-priori bound on  $\|\pt_\theta \Ub(t) \|_{L^\infty}$ would be sufficient to conclude global existence. However, by Peller's result and the Lax structure, we can only bound $\| \pt_\theta \Ub(t) \|_{L^1} \lesssim \| \Ub(0) \|_{B^1_{1,1}}$.}

Finally, we mention the remarkable feature that the half-wave maps equation preserves {\em rationality of solutions}. That is, if the initial datum $\Ub_0 : \T \to \Gr_k(\C^d)$ is a rational function of $z = \eu^{i \theta}$ with $\theta \in \T$, then the corresponding smooth solution $\Ub \in C(I; H^\infty)$ is a rational function $\Ub(t) : \T \to \Gr_k(\C^d)$ for any $t \in I$ on its maximal time interval $I$. Indeed, this feature follows from the unitary equivalence \eqref{eq:K_unitary} by the Lax evolution together with the fact that
$$
\mbox{$K_{\Ub(t)}$ has finite rank if and only if $\Ub(t) : \T \to \Gr_k(\C^d)$ is rational},
$$
which is a direct consequence of {\em Kronecker's theorem} for the Hankel operator $H_{\Ub(t)}$. In fact, the preservation of rationality will be our starting point for proving global-in-time existence of rational solutions of \eqref{eq:HWM}, which later will be used to construct the aforementioned flow map with data in $H^{1/2}$. 

\subsection*{Acknowledgments}
P.~G\'erard was partially supported by the French Agence Nationale de la Recherche under the ANR project ISAAC–ANR-23–CE40-0015-01. E.~Lenzmann gratefully acknowledges financial support from Swiss National Science Foundation (SNSF) through Grant No.~204121.

\section{Main results}

\label{sec:main}

In this section, we will state our main results that show global well-posedness of \eqref{eq:HWM} in the critical energy space $H^{1/2}(\T; \Gr_k(\C^d))$ and their almost periodicity in time.

\subsection{Global well-posedness in $H^{1/2}$}
As a starting point, we observe that an adaptation of the strategy in \cite{GeLe-25} yields the following global well-posedness result for all rational initial data, which we denote by $\mathcal{R}at(\T; \Gr_k(\C^d))$ henceforth.

\begin{lem}[GWP for rational data]
Let $d \geq 2$ and $1 \leq k \leq d-1$ be integers. Then, for every $\Ub_0 \in \mathcal{R}at(\T; \Gr_k(\C^d))$, there exists a unique solution $\Ub \in C^\infty(\R \times \T)$ of \eqref{eq:HWM} with initial datum $\Ub(0) =\Ub_0$. 

Moreover, we have conservation of the mean and the energy, i.e.,
$$
\Mean[\Ub(t)] = \Mean[\Ub_0] = \frac{1}{2 \pi} \int_\T \Ub_0, \quad E[\Ub(t)] = E[\Ub_0] \quad \mbox{for all $t \in \R$},
$$
and it holds that $\Ub(t) \in \mathcal{R}at(\T; \Gr_k(\C^d))$ for all $t \in \R$.
\end{lem}

Next, we recall that $\mathcal{R}at(\T; \Gr_k(\C^d))$ is dense in $H^{1/2}(\T; \Gr_k(\C^d))$, see Theorem \ref{thm:dense} below from \cite{GeLe-25}. Thus by taking limits of rational solutions, we easily construct  global weak solutions
$$
\Ub \in C(\R; H^{1/2}_{\mathrm{w}}) \quad \mbox{with} \quad E[\Ub(t)] \leq E[\Ub_0] \quad \mbox{for all $t \in \R$}
$$ 
of \eqref{eq:HWM} for initial data $\Ub_0 \in H^{1/2}(\T; \Gr_k(\C^d))$, where $H^{1/2}_{\mathrm{w}}$ denotes $H^{1/2}$ equipped with its weak topology. 

Our first main result provides a substantial strengthening by showing uniqueness of this weak limit and energy conservation, i.e., we have the equality 
$$
E[\Ub(t)] = E[\Ub_0] \quad \mbox{for all $t \in \R$},
$$ 
which yields the strong continuity of $\Ub \in C(\R; H^{1/2})$. We  formulate this as follows.

\begin{thm}[GWP in $H^{1/2}$] \label{thm:gwp}
Let $d \geq 2$ and $1 \leq k \leq d-1$ be integers. Then  \eqref{eq:HWM} with target $\Gr_k(\C^d)$ is \textbf{globally well-posed} in $H^{1/2}$ in the following sense. There exists a unique continuous map 
 $$
\Phi : \R \times H^{1/2}(\T; \Gr_k(\C^d)) \to H^{1/2}(\T; \Gr_k(\C^d)), \quad (t,\Ub_0) \mapsto \Phi_t(\Ub_0)
 $$ 
 such that, for every  $\Ub_0 \in H^{1/2}(\T; \Gr_k(\C^d))$, the following properties hold.
\begin{enumerate}
\item[(i)] \underline{Solution}: The map $t \mapsto \Phi_t(\Ub_0)$ belongs to $C(\R; H^{1/2}(\T; \Gr_k(\C^d)))$ and solves \eqref{eq:HWM} in the weak sense with initial datum $\Phi_{t=0}(\Ub_0)= \Ub_0$.
\item[(ii)] \underline{Mean and energy conservation}: For all $t \in \R$, it holds that
$$
\Mean[\Phi_t(\Ub_0)] = \Mean[\Ub_0] , \quad E[\Phi_t(\Ub_0)] = E[\Ub_0].
$$
\item[(iii)] \underline{Continuous dependence}: If $\Ub_{0,n} \in H^{1/2}(\T; \Gr_k(\C^d))$ is a sequence with $\Ub_{0,n} \to \Ub_0$ in $H^{1/2}$ as $n \to \infty$, then $\Phi_t(\Ub_{0,n}) \to \Phi_t(\Ub_0)$ in $H^{1/2}$ as $n \to \infty$ locally uniformly in time. 
\item[(iv)] \underline{Group property}: It holds that
$$
\Phi_{t+s}(\Ub_0) = \Phi_t(\Phi_s(\Ub_0)) \quad \mbox{for all $t,s \in \R$}.
$$
\item[(v)] \underline{Preservation of rationality}: For rational initial data $\Ub_0 \in \mathcal{R}at(\T; \Gr_k(\C^d))$, we have $\Phi_t(\Ub_0) \in \mathcal{R}at(\T; \Gr_k(\C^d))$ for all $t \in \R$.
\end{enumerate}
\end{thm}

Based on the previous theorem, we will refer to the data-to-solution map $\Phi$ above as the {\em flow map} for \eqref{eq:HWM} with initial data in $H^{1/2}$. 

The next result shows that the corresponding global-in-time solutions $\Ub(t) = \Phi_t(\Ub_0)$ from above still satisfy a Lax evolution, initially found for smooth solutions on short time intervals. Note that, due to the low regularity of solutions, the Lax equation \eqref{eq:Lax_intro} in commutator form becomes meaningless as it is unclear how to define $B_{\Ub(t)}$ if we only assume $\Ub(t) \in H^{1/2}$. However, we can still deduce that the Toeplitz operator $T_{\Ub(t)}$ evolves by a unitary equivalence in the sense of a Lax evolution. 

\begin{thm}[Lax evolution] \label{thm:lax}
Denote $\Ub(t) = \Phi_t(\Ub_0)$ for $\Ub_0 \in H^{1/2}(\T; \Gr_k(\C^d))$ with $\Phi$ as in Theorem \ref{thm:gwp} above. Then, for all $t \in \R$, there exists a unitary map $\Us(t) : L^2_+(\T; \C^{d \times d}) \to L^2_+(\T; \C^{d \times d})$ such that
$$
T_{\Ub(t)} = \Us(t) T_{\Ub_0} \Us(t)^*
$$
with the Toeplitz operator $T_{\Ub(t)} : L^2_+(\T; \C^{d \times d}) \to L^2_+(\T; \C^{d \times d})$ defined in \eqref{def:Toep}.
\end{thm}

\begin{remarks*}
1) We will prove Theorem \ref{thm:lax} before completing the proof of Theorem \ref{thm:gwp} below, as it provides us with a convenient way for proving energy conservation via the identity
$E[\Ub(t)] = \Tr(K_{\Ub(t)})$ together with the unitary equivalence for $K_{\Ub(t)} = \Us(t) K_{\Ub_0} \Us(t)^*$ for the trace-class operator $K_{\Ub(t)} = \id- T_{\Ub(t)}^2$. 

2) The unitary map $\Us(t)$ is given by a so-called {\em explicit formula} which we will discuss in Section \ref{sec:road} below. Also, the explicit formula for $\Us(t)$ will play a key role in constructing the flow map $\Phi$ in Theorem \ref{thm:gwp} above.
\end{remarks*}

\subsection{Long-time behavior of solutions}
Our second main result shows that the solutions for (HWM) obtained by Theorem \ref{thm:gwp} are almost periodic in time; see  Section \ref{sec:ap} for a brief recap of almost periodic functions valued in Banach spaces.

\begin{thm}[Almost periodicity] \label{thm:ap}
Let $d \geq 2$ and $1 \leq k \leq d-1$ be integers. For every  $\Ub_0 \in H^{1/2}(\T; \Gr_k(\C^d))$, the map $t \mapsto \Phi_t(\Ub_0) \in C(\R; H^{1/2})$  is \textbf{almost periodic} in time. As a consequence, we have the following properties.
\begin{enumerate}
\item[(i)]  Poincar\'e recurrence holds in the sense that, for all $\eps, T > 0$, there exists some time $t_* \geq T$ such that $\| \Phi_{t_*}(\Ub_0) - \Ub_0 \|_{H^{1/2}} \leq \eps$.
\item[(ii)] The orbit $\{ \Phi_t(\Ub_0) : t \in \R \}$ is relatively compact in $H^{1/2}$.
\end{enumerate}
\end{thm}

In the case of rational initial data, we can actually obtain the following refinement by showing quasi-periodicity in time, which corresponds to a linear flow on some torus $\T^N=(\R/2\pi \Z)^N$. In particular, we obtain a-priori bounds on all Sobolev norms for rational solutions of \eqref{eq:HWM}. Here is the precise statement.

\begin{thm}[Quasi-periodicity and a-priori bounds for rational data] \label{thm:rat_qp}
Let $d \geq 2$ and $1 \leq k \leq d-1$ be integers. For every rational initial datum $\Ub_0 \in \mathcal{R}at(\T; \Gr_k(\C^d))$, the map $t \mapsto \Phi_t(\Ub_0)$ is \textbf{quasi-periodic} in time. That is, it holds that
$$
\Phi_t(\Ub_0) = \mathbf{G}( t \bm{\omega} ) \quad \mbox{for all $t \in \R$},
$$
where $\mathbf{G} : \T^N \to \mathcal{R}at(\T; \Gr_k(\C^d))$ is some map depending only on $\Ub_0$ with some integer $N \geq 1$ and some constant $\bm{\omega} \in \R^N$. 

In particular, we obtain the a-priori bounds
$$
\sup_{t \in \R} \| \Phi_t(\Ub_0) \|_{H^s} \lesssim_{\Ub_0, s} 1 \quad \mbox{for all $s \geq 0$}.
$$
\end{thm}

\begin{remark*}
Note that the a-priori bounds on all $H^s$-norms for rational solutions of \eqref{eq:HWM} obtained above are a consequence of their quasi-periodicity in time -- and not deduced from a family of conserved quantities. A natural open question is to gain insight into the behavior of $H^s$-norms with $s > \frac 1 2$ for non-rational initial data. More ambitiously, it would be interesting to prove (or disprove) global well-posedness of \eqref{eq:HWM} in $H^s$ for $s> \frac 1 2$.
\end{remark*}

\section{Road map and remarks}

\label{sec:road}

\subsection{GWP for rational data}

As the initial step towards proving Theorem \ref{thm:gwp}, we consider rational initial data $\Ub_0 \in \mathcal{R}at(\T; \Gr_k(\C^d))$. Using the local well-posedness in $H^s$ for any $s > \frac 3 2$, we let $\Ub \in C(I; H^s)$ denote the corresponding solution of \eqref{eq:HWM} with $\Ub(0) = \Ub_0$ defined on its maximal time interval of existence $I$. Now, the analysis of the Lax pair structure on the Hardy space $L^2_+=L^2_+(\T; \C^{d \times d})$ allows to deduce that $\Ub$ can be expressed in terms of the {\em explicit formula} given by
\be \tag{EF} \label{eq:EF}
\boxed{\Pi \Ub(t,z) = \Mean \left ( ( \id - z \eu^{-itT_{\Ub_0}} S^*)^{-1} \Pi \Ub_0 \right ) \quad \mbox{for} \quad (t,z)  \in I \times \D}
\ee
where $T_{\Ub_0}$ denotes the Toeplitz operator defined in \eqref{def:Toep}. Here we recall that $\Pi$ denotes the Cauchy--Szeg\H{o} projection onto $L^2_+$ and the operator $\Mean$ is the mean, whereas $S^*$ stands for the backward (or right) shift operator on $L^2_+$. To further clarify the meaning of \eqref{eq:EF}, we remark that we use the canonical identification of elements $\Fb \in L^2_+(\T; \C^{d \times d})$ with holomorphic functions $\Fb : \D \to \C^{d \times d}$ with $\Fb(z) = \sum_{n \geq 0} \widehat{\Fb}_n z^n$ such that $\sum_{n \geq 0}|\widehat{\Fb}_n|_E^2 < \infty$, where $|\cdot|_E$ denotes the Frobenius norm of matrices in $\C^{d \times d}$; see Section \ref{sec:prelim} below for a recap on vector-valued Hardy spaces. The interested reader will also find  below some discussion on the relation of \eqref{eq:EF} to other explicit formulae that have been found for completely integrable PDEs, which was initiated by the work of the first author of the present paper.

To show global existence, i.e., we have that $I = \R$ holds, it suffices to obtain an a-priori bound
\be \label{ineq:apriori_intro}
\sup_{t \in I} \| \Ub(t) \|_{H^s} < \infty
\ee
for any given $s > \frac 3 2$. To achieve this goal, we can adapt the strategy in \cite{GeLe-25} for the half-wave maps equation $\R$ for rational initial data. More precisely, it turns out that the rationality of $\Ub_0 \in \mathcal{R}at(\T; \Gr_k(\C^d))$ in concert with Kronecker's theorem for Hankel operators allows us to single out a {\em finite-dimensional} subspace $\Ks \subset L^2_+$ with the following properties. 
\begin{itemize}
\item[1)] $\Pi \Ub_0 \in \Ks$.
\item[2)] $(\id - z \eu^{-it T_{\Ub_0}} S^*)^{-1} : \Ks \to \Ks$ for all $t \in \R$ and $z \in \D$.
\end{itemize}
Hence we see that the analysis of \eqref{eq:EF} for rational data completely reduces to matters in some finite-dimensional subspace (and the corresponding dynamics could be described by a finite system of nonlinear ODEs). In particular, the question of obtaining a-priori bounds boils down to ruling out potential eigenvalues of the contraction $\eu^{-itT_{\Ub_0}} S^* : \Ks \to \Ks$ on the unit circle for all times, i.e.,
\be \label{eq:eigen_T}
\sp(\eu^{-it T_{\Ub_0}} S^* |_\Ks) \cap \pt \D = \emptyset \quad \mbox{for all $t \in \R$}.
\ee 
Indeed, we can derive \eqref{eq:eigen_T} by exploiting commutator properties of $T_{\Ub_0}$ with the backward shift $S^*$. Once the spectral property \eqref{eq:eigen_T} has been established, we obtain strong uniform bounds on the resolvent $(\id - z \eu^{-it T_{\Ub_0}} S^*)^{-1}$ on $\Ks$. By going back to \eqref{eq:EF}, this allows us to deduce the desired a-priori bound \eqref{ineq:apriori_intro} and we thus obtain global existence for all rational data. In fact, a slightly more refined analysis of \eqref{eq:EF} yields the quasi-periodicity for rational solutions in Theorem \ref{thm:rat_qp} above.

\subsection{GWP in $H^{1/2}$}
For non-rational data $\Ub_0 \in H^{1/2}(\T; \Gr_k(\C^d))$, we  choose a sequence $\Ub_{0,n} \in \mathcal{R}at(\T; \Gr_k(\C^d))$ such that $\Ub_{0,n} \to \Ub_0$ in $H^{1/2}$ thanks to the density result in Theorem \ref{thm:dense}. Let $\Ub_n \in C^\infty(\R \times \T)$ denote the corresponding global smooth solutions of \eqref{eq:HWM} with initial datum $\Ub_n(0) = \Ub_{0,n}$. Next, by invoking the explicit formula \eqref{eq:EF} valid for each $\Ub_n(t)$ and by using energy conservation for these smooth solutions, we can deduce that
\be \label{eq:conv_U}
\mbox{$\Ub_n(t) \to \Ub(t)$ in $L^2$ and $\Ub_n(t) \weakto \Ub(t)$ in $H^{1/2}$ for any $t \in \R$}.
\ee
Here the limit $\Ub \in C(\R; H^{1/2}_{\mathrm{w}})$ is a weak solution of \eqref{eq:HWM} with $\Ub(0) = \Ub_0$ and it is given by the explicit formula
\be \tag{EF}
\Pi \Ub(t,z) = \Mean \left ( ( \id - z \eu^{-it T_{\Ub_0}} S^*)^{-1} \Pi \Ub_0 \right ) \quad \mbox{for $(t,z) \in \R \times \D$},
\ee
which in particular shows that the limit $\Ub(t)$ is independent of the chosen sequence of approximating rational data $\Ub_{0,n}$. However, the weak convergence in $H^{1/2}$ only guarantees the inequality
$$
E[\Ub(t)] \leq E[\Ub_0] \quad \mbox{for all $t \in \R$}.
$$
Now, the major key step towards proving global well-posedness in the sense of Theorem \ref{thm:gwp} rests on showing that energy conservation holds, i.e., we have 
$$
E[\Ub(t)] = E[\Ub_0] \quad \mbox{for all $t \in \R$}
$$
and thus the limit $\Ub \in C(\R; H^{1/2})$ is in fact continuous.

However, the analysis of \eqref{eq:EF} for non-rational data $\Ub_0 \in H^{1/2}$ is substantially more complicated, since the analogue of the subspace $\Ks \subset L^2_+$ from before must be {\em infinite-dimensional} due to Kronecker's theorem. In particular, the spectral property \eqref{eq:eigen_T} -- although also true for non-rational $\Ub_0$ -- is far from sufficient to get enough control from \eqref{eq:EF} directly. To overcome this obstacle, we devise the following strategy. 

Given a sequence $\Ub_{0,n} \in \mathcal{R}at(\T; \Gr_k(\C^d))$ with $\Ub_{0,n} \to \Ub_0$ in $H^{1/2}$, we introduce the following maps $\Us_n(t) : L^2_+ \to L^2_+$ by setting
\be
(\Us_n(t) \Fb)(z) := \Mean  \big ( (\id - z \eu^{-it T_{\Ub_{0,n}}} S^*)^{-1} \Fb \big ) 
\ee
for $t \in \R$, $z \in \D$, and $\Fb \in L^2_+$. Likewise, we define the map $\Us(t) : L^2_+ \to L^2_+$ as
\be \label{def:Us_intro}
(\Us(t) \Fb)(z) := \Mean \left ( (\id - z \eu^{-it T_{\Ub_{0}}} S^*)^{-1} \Fb \right ).
\ee
By the Lax evolution for the smooth rational solutions $\Ub_n \in C^\infty(\R \times \T)$, we can deduce the following properties for all $t \in \R$ and $n \in \N$:
\begin{enumerate}
\item[1)] $\pt_t \Us_n(t) = B_{\Ub_n(t)} \Us_n(t)$ with $\Us_n(0) = \id$.
\item[2)] $\Us_n(t) : L^2_+ \to L^2_+$ is unitary. 
\item[3)] $T_{\Ub_n(t)} \Us_n(t) = \Us_n(t) T_{\Ub_{0,n}}$.
\end{enumerate}
Note that the convergence properties in \eqref{eq:conv_U} imply that $T_{\Ub_n(t)} \to T_{\Ub(t)}$ strongly as operators for any $t \in \R$, which is seen to yield, for every $t \in \R$, that 
$$
\mbox{$\Us_n(t) \weakto \Us(t)$ weakly as operators}.
$$ 
Moreover, by using the strong operator convergence of the self-adjoint operators $T_{\Ub_n(t)}$, we can pass to the limit in (3) to deduce the intertwining relation
\be \label{eq:intertwining_intro}
T_{\Ub(t)} \Us(t) = \Us(t) T_{\Ub_0} \quad \mbox{for all $t \in \R$}.
\ee
Now the central question is whether the map $\Us(t) : L^2_+ \to L^2_+$ is {\em unitary} for all $t \in \R$? If so, then we can use \eqref{eq:intertwining_intro} to infer the unitary equivalence
$$
K_{\Ub(t)} = \Us(t) T_{\Ub_0} \Us(t)^* \quad \mbox{for all $t \in \R$}
$$
for the trace-class operator $K_{\Ub(t)} = \id - T_{\Ub(t)}^2$. Once this is established, we can directly deduce the desired energy conservation in view of $E[\Ub(t)] = \Tr (K_{\Ub(t)}) = \Tr(K_{\Ub_0}) = E[\Ub_0]$.

To establish that the limiting map $\Us(t)$ is indeed unitary, we derive the following {\em stability principle} for explicit formulae in a general setting that goes beyond the specifics of \eqref{eq:HWM}. 

\begin{thmn}[Stability principle for explicit formulae]
Let $E$ be a complex Hilbert space and suppose that $L : \dom(L) \subset L^2_+(\T; E) \to L^2_+(\T; E)$ is a (possibly unbounded) self-adjoint operator. For $t \in \R$, we define the map $\Us_L(t) : L^2_+(\T; E) \to L^2_+(\T;E)$ by setting
$$
(\Us_L(t) \Fb)(z) := \Mean \left ( ( \id- z \eu^{-it L} S^*)^{-1} \Fb \right ) \quad \mbox{with $z \in \D$}.
$$
Then $\Us_L(t)$ is unitary for all $t \in \R$ if and only if one of the following equivalent conditions holds.
\begin{enumerate}
\item[(i)] $\ker \, \Us_L(t) = \{ 0 \}$ for all $t\in \R$.
\item[(ii)] $\displaystyle \lim_{j \to \infty} \| (\eu^{-it L} S^*)^j \Fb \|_{L^2} = 0$ for all $t \in \R$ and $\Fb \in L^2_+(\T; E)$. 
\end{enumerate}
\end{thmn} 

\begin{remarks*}
1) The proof of this theorem rests on some operator-theoretic analysis, which connects to the circle of ideas of the Wold decomposition for isometries on Hilbert spaces. However, we will assume no operator-theoretic background and provide a self-contained approach.

2) Condition (ii) means that the discrete semigroup $\{ (\eu^{-it L} S^*)^j \}_{j \in \N}$ is {\em strongly stable} for each $t \in \R$. This is a natural condition from an abstract operator point of view. Owed to this fact, we have chosen to refer to the result above as the {\em stability principle} for explicit formulae.

3) The triviality of the kernel of $\Us_L(t)$ stated in (i) can be seen as a sort of {\em nondegeneracy} condition. In the context of explicit formulae for completely integrable PDEs, this nondegeneracy seems easier to be verified than (ii). See in particular the proof below, where we establish that $\ker \, \Us(t) = \{ 0 \}$ in the context of (HWM).

4) However, we remark that we can also find examples where $\ker \, \Us(t) \neq \{ 0 \}$ occurs for some $t \in \R$ with a suitable choice for the self-adjoint operator $L$. For instance, this occurs for the explicit formula for the {\em zero-dispersion limit} for the Benjamin--Ono (BO) equation, which corresponds to breakdown (shock formation) of solutions in finite time. In Subsection \ref{subsec:ZD_BO} below, we compare  the explicit formulae for the half-wave maps equation and the zero-dispersion limit of (BO).

5) In a companion work on global well-posedness for \eqref{eq:HWM} on the real line, we derive an analogous stability principle for explicit formulae in the Hardy space $L^2_+(\R;E)$.
\end{remarks*}

In view of the general result above, it remains to show that $\Us(t)$ defined in \eqref{def:Us_intro} satisfies $\ker \, \Us(t) = \{ 0 \}$ for all $t \in \R$. Let us briefly sketch the proof as follows. From the intertwining relation \eqref{eq:intertwining_intro} we see that eigenvectors $\Fb \in L^2_+$ of $T_{\Ub_0}$ are transported by $\Us(t)$: If $T_{\Ub_0} \Fb = \mu \Fb$ for some eigenvalue $\mu$, then $\Fb(t) := \Us(t) \Fb$ satisfies the equation
$$
T_{\Ub(t)} \Fb(t) = \mu \Fb(t).
$$
However, the crucial step is now to show that $\Fb(t)$ is actually an eigenvector of $T_{\Ub(t)}$ as well, i.e., we have $\Fb(t) \neq 0$. Here, the specifics of the Lax operator $T_{\Ub_0}$ for \eqref{eq:HWM} enter, where we make use of the crucial property that eigenspaces of Toeplitz operators are {\em nearly $S^*$-invariant}\footnote{The nearly $S^*$-invariance property can be seen as a substitute for the lack of elliptic regularity for eigenfunctions of the zero-order operator $T_{\Ub_0}$. By contrast, elliptic regularity for eigenfunctions of the first-order Lax operators $L$ for (BO) and (CS-DNLS) on $\T$ can be used to gain control of Sobolev norms in the transport by $\Us(t)$ to rule out the vanishing of $\Fb(t)=\Us(t) \Fb$; see \cite{Ba-24}.}. We refer to Section \ref{sec:gwp} for details. 

Finally, once we have shown that $\Us(t)$ maps eigenfunctions of $T_{\Ub_0}$ to eigenfunctions of $T_{\Ub(t)}$, the triviality of $\ker \, \Us(t) = \{ 0 \}$ can be derived by using that the spectrum of $T_{\Ub_0}$ is pure point, i.e., there is a basis of $L^2_+$ consisting of eigenfunctions of $T_{\Ub_0}$ together with the fact that $T_{\Ub(t)}$ is self-adjoint.

\subsection{Almost periodicity} We now discuss how to prove almost periodicity of the solution $\Ub(t) = \Phi_t(\Ub_0) \in C(\R; H^{1/2})$ with initial datum $\Ub_0 \in H^{1/2}(\T; \Gr_k(\C^d))$. By Bochner's criterion for almost periodicity, it suffices to show that the set of translates of $\Ub \in C(\R;H^{1/2})$ defined as
$$
\mathrm{Trans}(\Ub) := \{ \Ub(\cdot + a) : a \in \R \}
$$
is relatively compact in the Banach space of bounded and continuous functions $\BC(\R; H^{1/2})$ equipped with the $\sup$-norm.

To show this property, we make use of the explicit formula \eqref{eq:EF} as follows. First, by the construction of the flow map $\Phi$, we have
$$
\Ub(t) = \Phi_t(\Ub_0) = (\Us(t) \Pi \Ub_0) + (\Us(t) \Pi \Ub_0)^* - \Mean (\Ub_0).
$$
For the sake of the proof, let us write the unitary map $\Us(t) : L^2_+ \to L^2_+$ as
$$
(\Us(t) \Fb)(z) = \Mean \left ( (\id- z \Om(t) S^*)^{-1} \Fb \right ) \quad \mbox{for $(t, z) \in \R \times \D$}
$$
with the strongly continuous unitary one-parameter group
$$
\Om(t) := \eu^{-it T_{\Ub_0}} : L^2_+ \to L^2_+ \quad \mbox{for $t \in \R$}.
$$
 Since $T_{\Ub_0}$ has at most countable pure point spectrum, a standard Cantor diagonalization argument yields the following {\em compactness property}:  For every sequence $(a_n)$ in $\R$, there exists a subsequence denoted by $(a'_n)$ such that 
$$
\mbox{$\Om(a'_n) \to \Om_\infty$  strongly as operators}
$$ 
with some unitary map $\Om_\infty$ in $L^2_+$. Next, by revisiting the analysis of the explicit formula in the global well-posedness proof combined with the group property of $\Om(t)$, we can show that
$$
\sup_{t \in \R} \| \Ub(t+a'_n) - \Ub_\infty(t) \|_{H^{1/2}} \to 0 \quad \mbox{as $n \to \infty$},
$$
where $\Ub_\infty \in C(\R; H^{1/2})$ is given by the explicit-type formula
$$
\Pi \Ub_\infty(t,z) = \Mean \left ( ( \id - z \Omega(t) \Omega_\infty S^*)^{-1} \Pi \Ub_0 \right ) \quad \mbox{for $(t,z) \in \R \times \D$}.
$$  
This chain of steps proves the relative compactness of $\mathrm{Trans}(\Ub)$ in $\BC(\R; H^{1/2})$.

As a final remark, we note that proving almost periodicity by the explicit formula only involves that the corresponding Lax operator $L$ generating $\Om(t) = \eu^{-it L}$ has at most countable pure point spectrum. Thus we expect our arguments to carry over to (BO) and (CS-DNLS) on $\T$.  

\subsection{On the zero-dispersion limit of (BO)}
\label{subsec:ZD_BO}
We remark that \eqref{eq:EF} has a striking resemblance to the explicit formula for the zero-dispersion limit for (BO), first derived in \cite{Ga-23a} (and adapted in \cite{Ma-25}) in the periodic setting, and in \cite{Ge-24} by the first author of this paper on the line. More precisely, we recall from \cite{Ga-23a, Ma-25} that for initial data $u_0 \in L^\infty(\T;\R)$ the zero-dispersion limit of (BO) denoted by $ZD[u_0](t) \in C(\R; L^2_{\mathrm{w}}(\T;\R))$ is given by the explicit formula (adapted to our notation) 
$$
\Pi (ZD[u_0])(t,z) = \Mean \left ( (\id - z \eu^{-it T_{u_0}} S^*)^{-1} \Pi u_0 \right ) \quad \mbox{with $(t,z) \in \R \times \D$},
$$
where $\Mean = \langle \cdot | 1 \rangle$ with the constant function 1 corresponds to the mean in $L^2_+(\T; \C)$. Here $T_{u_0} : L^2_+(\T; \C) \to L^2_+(\T;\C)$ is the Toeplitz operator with real-valued symbol $u_0 \in L^\infty(\T;\R)$. However, it is known that $ZD[u_0]$ can exhibit loss of $L^2$-norm after the first breaking time 
$$
T_+ = - \frac{1}{2 \min_{\theta \in \T} u_0'(\theta)} > 0
$$
for any non-constant initial datum $u_0 \in C^1(\T; \R)$, which corresponds to the onset of a shock at $T_+>0$ for the inviscid Burger's equation $\pt_t u + 2 u \pt_\theta u = 0$ with $u |_{t=0}=u_0$. Indeed, for $u_0(\theta) = -\cos  \theta$ and hence $T_+ = \frac{1}{2}$, we deduce from \cite{Ma-25} together with \cite{Ga-23}[Lemma 20] the strict inequality
$$
\| \Pi (ZD[u_0])(t) \|_{L^2} < \| \Pi u_0 \|_{L^2} \quad \mbox{for} \quad T_+ < t < T_+ + \eps
$$ 
with some $\eps > 0$ sufficiently small. Hence the explicit formula above fails to yield a unitary map on $L^2_+$ right after the breaking time $T_+> 0$. In view of the stability theorem above, this means that the discrete semigroup $\{ A_t^j \}_{j \in \N}$ with the contraction
$$
A_t = \eu^{-it T_{\cos \theta}} S^*  = \eu^{-\frac{i}{2} t (S + S^*)} S^*
$$ 
is not strongly stable for $T_+ < t < T_+ + \eps$. From an operator-analytic point of view, it would be interesting to prove this failure of strong stability for $A_t$ directly, without resorting to the shock formation above.

In view of the preceding discussion, it is a remarkable fact that we can rule out formation of shocks for \eqref{eq:HWM} in contrast to the zero-dispersion limit for (BO). As a matter of fact, the deeper reason lies in the entirely different spectral properties of the Toeplitz operators $T_{u_0}$ and $T_{\Ub_0}$. More precisely, for all non-constant $u_0 \in C(\T; \R)$ and any $\Ub_0 \in H^{1/2}(\T; \Gr_k(\C^d))$, it holds that:
\begin{itemize}
\item $T_{u_0}$ has no point spectrum $\sp(T_{u_0}) = \emptyset$ but only purely a.c.-spectrum.
\item $T_{\Ub_0}$ has only pure point spectrum $\sigma(T_{\Ub_0}) = \sp(T_{\Ub_0})$.
\end{itemize}
In particular, these spectral properties of $T_{\Ub_0}$ allow us to establish the unitarity of the map $\Us(t)$ in the explicit formula for \eqref{eq:HWM}.

\section{Notation and preliminaries}

\label{sec:prelim}

\subsection{Vector-valued Hardy spaces and shift operators}

We recall some basic facts and notions from the theory of $L^2$-based Hardy spaces of functions taking values in a given Hilbert space, also referred to as vector-valued Hardy spaces. The interested reader may consult \cite{SzFoBeKe-10, Pe-03} for more background.

Let $E$ be a complex (not necessarily separable) Hilbert space with inner product $\langle \cdot | \cdot \rangle_E$ and corresponding norm $|\cdot |_E$.  We denote by $L^2(\T; E)$ the Hilbert space of strongly measurable and square-integrable functions on $\T$ valued in $E$ with the inner product 
$$
\langle \Fb | \Gb \rangle = \frac{1}{2 \pi} \int_\T \langle \Fb(\theta) | \Gb(\theta) \rangle_E \, d \theta .
$$
By Parseval's theorem, we notice that
$$
 L^2(\T; E) = \big \{ \Fb = \sum_{n=-\infty}^\infty \widehat{\Fb}_n \eu^{in \theta} : \sum_{n=-\infty}^\infty |\widehat{\Fb}_n|_E^2 < +\infty \big \},
$$
where $\widehat{\Fb}_n = \frac{1}{2 \pi} \int_0^{2\pi} \Fb(\theta) \, \eu^{-in \theta} \, d\theta \in E$ denotes the $n$-th Fourier coefficient of $\Fb$. 

Next, we let $\Pi$ denote the Cauchy--Szeg\H{o} projection onto the closed subspace of elements in $L^2(\T; E)$ whose negative Fourier coefficients vanish, that is, we set
$$
\Pi \left ( \sum_{n=-\infty}^\infty \widehat{\Fb}_n \eu^{int} \right ) = \sum_{n=0}^\infty \widehat{\Fb}_n \eu^{int}.
$$
The range of $\Pi$ is the $L^2$-based Hardy space of functions valued in $E$ such that 
$$
L^2_+(\T;E) = \Pi(L^2(\T;E)) = \big \{ \Fb =\sum_{n=0}^\infty \widehat{\Fb}_n \eu^{int} : \widehat{\Fb}_n \in E, \; \sum_{n=0}^\infty |\widehat{\Fb}_n|^2_E < \infty  \big \},
$$
which is of course a Hilbert space in its own right. Note that $\Pi$ is an orthogonal projection and we denote by $\Pi_- =  \id - \Pi$ its complementing orthogonal projection with range
$$
L^2_-(\T; E) = \Pi_-(L^2(\T;E) = \big \{ \Fb =\sum_{n <0} \widehat{\Fb}_n \eu^{int} :  \widehat{\Fb}_n \in E, \; \sum_{n <0} |\widehat{\Fb}_n|^2_E < \infty \big  \}.
$$ 
To streamline our notation, we will henceforth often use the shorthand notation
$$
L^2 = L^2(\T; E) \quad \mbox{and} \quad L_{\pm}^2 = L_{\pm}^2(\T; E) 
$$ 
whenever the choice of $E$ is clear from the context. 

It is a classical fact that elements $\Fb \in L^2_+(\T; E)$ can be identified with holomorphic functions $\Fb : \D \to E$ such that 
$$
\sup_{r < 1} \int_0^{2 \pi} |\Fb(r \eu^{i \theta})|^2_E \, d\theta < +\infty,
$$
in which case the equality $\| \Fb \|_{L^2}^2 = \frac{1}{2\pi} \sup_{r < 1} \int_0^{2 \pi} |\Fb(r \eu^{i\theta})|^2_E \, d \theta$ holds and the radial limit $\lim_{r \to 1^{-}} \Fb(r \eu^{i \theta})$ exists for a.e.~$\theta \in \T$. By adapting common practice in the theory of Hardy spaces, we will freely make use of this canonical identification of $\Fb \in L^2_+(\T;E)$ with the corresponding holomorphic function $\Fb : \D \to E$ satisfying the condition stated above. 

On the Hardy space $L^2_+$, there is a canonical operator $S \in \Ls(L^2_+)$ together with its adjoint $S^* \in \Ls(L^2_+)$, which are given by
$$
S \Fb= S \left ( \sum_{n=0}^\infty \widehat{\Fb}_n \eu^{in \theta} \right ) := \sum_{n=1}^\infty \widehat{\Fb}_{n-1}  \eu^{in\theta}, 
$$
$$
S^* \Fb= S^*\left ( \sum_{n=0}^\infty \widehat{\Fb}_n \eu^{in \theta} \right ) = \sum_{n=0}^\infty \widehat{\Fb}_{n+1} \eu^{in \theta}.
$$
The bounded operators $S$ and $S^*$ are usually referred to as the {\em forward shift} (or right shift) and the {\em backward shift} (or left shift) on $L^2_+$, respectively.  We easily note that, for any $\Fb \in L^2_+$ when regarded as a holomorphic function $\Fb=\Fb(z)$ with $z \in \D$, we can equivalently  write
$$
(S \Fb)(z) = z \Fb(z), \quad (S^* \Fb)(z) = \frac{\Fb(z)-\Fb(0)}{z}.
$$
In particular, the backward shift $S^*$ will play a central role in the analysis of (HWM) on $\T$. Moreover, we recall the fundamental identities
\be \label{eq:Sdefect}
S^* S \Fb = \Fb \quad \mbox{and} \quad S S^* \Fb = \Fb - \Mean(\Fb),
\ee
where 
$$
 \Mean(\Fb) =\frac{1}{2 \pi} \int_0^{2 \pi} \Fb(\eu^{i \theta}) \, d\theta = \widehat{\Fb}_0.
$$
is the orthogonal projection onto the subspace of constant functions $E \subset L^2_+(\T; E)$ corresponding to taking the mean (or equivalently the $0$-th Fourier coefficient) of $\Fb \in L^2_+$. Since $\Mean((S^*)^k \Fb) = \widehat{\Fb}_k$ for all $k \in \N$, a geometric series expansion yields the identity
\be \label{eq:reproduce}
\boxed{\Fb(z) = \Mean \left ( ( \id - z S^*)^{-1} \Fb \right )}
\ee 
valid for all $\Fb \in L^2_+(\T; E)$ and $z \in \D$. In fact, we can regard \eqref{eq:reproduce} as a {\em reproducing kernel formula} for the Hardy space $L^2_+(\T;E)$.

In our analysis of \eqref{eq:HWM}, we will mostly consider the Hilbert space $E =\C^{d \times d}$ of the complex $d \times d$-matrices equipped with the standard (Frobenius) inner product 
$$
\langle A, B \rangle_E = \Tr (A B^*) = \sum_{j,k=1}^d A_{jk} \ov{B}_{jk}
$$
together with its corresponding norm $|A|_E = \langle A, A \rangle_E^{1/2}$ for matrices $A \in \C^{d \times d}$.

\subsection{Toeplitz and Hankel operators} Recall that $E$ is a complex Hilbert space. Let $\Ls(E)$ denote the Banach space of bounded linear operators on $E$ into itself equipped with the operator norm. We use $L^\infty(\T; \Ls(E))$ to denote the Banach space of (Bochner) measurable maps from $\T$ to $\Ls(E)$ with finite norm given by
$$
\| \Ub \|_{L^\infty} := \mathrm{ess \, sup}_{\theta \in \T} \| \Ub(\theta) \|_{\Ls(E)}.
$$
Again, we remark that we will be mainly interested in the choice $E = \C^{d \times d}$. Note that every  $U \in \C^{d \times d}$ can be naturally seen as an element in $\Ls(E)$ by acting on  elements on $E=\C^{d \times d}$ by left matrix multiplication. In particular, we see that $L^\infty(\T; \C^{d \times d}) \subset L^\infty(\T; \Ls(E))$ when $E=\C^{d \times d}$.
 
For given $\Ub \in L^\infty(\T; E)$, we define the corresponding {\em Toeplitz operator} with the symbol $\Ub$ by setting
$$
T_{\Ub} : L^2_+(\T;E) \to L^2_+(\T;E), \quad \Fb \mapsto T_\Ub \Fb := \Pi(\Ub \Fb). 
$$  
Likewise, we define the {\em Hankel operator} with symbol $\Ub$ to be the operator
$$
H_{\Ub} : L^2_+(\T; E) \to L^2_-(\T;E), \quad \Fb \mapsto H_{\Ub} \Fb := \Pi_-(\Ub \Fb).
$$
As an aside, we inform the reader that there exist alternative (but ultimately equivalent) ways of defining Hankel operators in the literature; e.g., as the anti-linear operator $\Fb \mapsto \Pi(U \ov{\Fb})$ or by using the unitary flip operator $J$ on $L^2$ with $J \Pi_- = \Pi J$. However, we prefer the definition of $H_\Ub$ above, as it turns out to be the most convenient one for the analysis of (HWM) below. In particular, we follow the definition of Hankel operators as done in \cite{Pe-03}.

Evidently, we have the operator norm bounds  with  $\| T_{\Ub} \| \leq \| \Ub \|_{L^\infty}$ and $\|H_{\Ub} \| \leq \|  \Ub \|_{L^\infty}$. Furthermore, it is readily checked that the adjoints of $T_{\Ub}^* : L^2_+ \to L^2_+$ and $H_{\Ub}^* : L^2_- \to L^2_+$ are found to be $T_\Ub^* f= \Pi (\Ub^* f)$ and $H_{\Ub}^* f = \Pi(\Ub^* f)$, where $\Ub^*$ denotes the pointwise adjoint of $\Ub : \T \to \Ls(E)$. In particular, we see that $T_{\Ub} = T_\Ub^*$ is self-adjoint if and only if $\Ub(\theta) = \Ub(\theta)^*$ is self-adjoint on $E$ for a.e.~$\theta \in \T$. In addition, we record two important facts about Hankel operators $H_\Ub$ with symbol $\Ub \in L^\infty(\T; E)$ that will be detailed and used in Section \ref{sec:spec} below.

\begin{itemize}
\item $H_\Ub$ is {\em Hilbert--Schmidt} if and only if $\sum_{n <0} |n| \| \widehat{\Ub}_n \|^2_{\Ls(E)} < \infty$. In the case $E=\C^{d \times d}$, the latter condition means that $\Pi_- \Ub$ belongs to the Sobolev space $H^{1/2}(\T; \C^{d \times d})$.
\item $H_\Ub$ has {\em finite rank} if and only if $\Pi_- \Ub$ is a rational function of $\eu^{i \theta}$ with $\theta \in \T$ with no poles on the unit circle.
\end{itemize}

Next, we recall the following commutator formula for Toeplitz operators $T_\Ub$ with the backward shift $S^*$. 

\begin{lem} \label{lem:commutator}
For $\Ub \in L^\infty(\T; \Ls(E))$ and $\Fb \in L^2_+(\T;E)$, it holds
$$
[S^*, T_\Ub] \Fb = S^* T_{\Ub}(\Mean(\Fb)).
$$

In particular, if  $E=\C^{d \times d}$ and $\Ub \in L^\infty(\T; \C^{d \times d})$, the identity above reads
$$
[S^*, T_\Ub] \Fb = (S^* \Pi \Ub)\Mean(\Fb),
$$
with the backward shift $S^*$ acting on $\Pi \Ub \in L^2_+(\T, \C^{d \times d})$.
\end{lem}

\begin{remark*}
 As a direct consequence of Lemma \ref{lem:commutator}, we deduce that eigenspaces of $T_\Ub$ are {\em nearly $S^*$-invariant}, which will turn out to be a key feature; see Section \ref{sec:spec} below.
\end{remark*}

\begin{proof}
The proof of Lemma \ref{lem:commutator} is elementary, yet we provide the details for the reader's convenience. Let $\Fb \in L^2_+(\T;E)$ be given. Using the identity $S^* T_\Ub S = T_\Ub$ (which is straightforward to check for any Toeplitz operator $T_\Ub$) in combination with identity \eqref{eq:Sdefect}, we find
$$
[S^*, T_{\Ub}] \Fb  =  S^* T_\Ub \Fb - T_\Ub S^* \Fb = S^*T_\Ub \Fb- S^* T_{\Ub} S S^* \Fb = S^*T_{\Ub}(\Mean(\Fb)).
$$
This shows the claimed identity. If $E=\C^{d \times d}$ and $\Ub \in L^\infty(\T; \C^{d \times d})$, we can write this identity as stated above, suing that $\Mean(\Fb) \in \C^{d \times d}$ is a constant matix.
\end{proof}

For later use, we also record the following key identity relating Toeplitz and Hankel operators:
\be \label{eq:ToeplitzHankel}
T_{\Fb \Gb} - T_{\Fb} T_{\Gb} = H^*_{\Fb^*} H_{\Gb}
\ee
valid for any $\Fb, \Gb \in L^\infty(\T; \Ls(E))$. The elementary proof of \eqref{eq:ToeplitzHankel} is left to the reader.

\subsection{Sobolev-type spaces and density of rational maps}
In what follows, we take $E= \C^{d \times d}$ with some fixed integer $d \geq 2$ equipped with standard (Frobenius) inner product denoted by $\langle A|B \rangle_E= \tr(A B^*)$ for $A,B \in \C^{d \times d}$.  For real $s \geq 0$, we define the Sobolev space for matrix-valued functions on $\T$ given by
$$
H^s := H^s(\T; \C^{d \times d}) := \{ \Fb \in L^2(\T; \C^{d \times d}) : \| \Fb\|_{H^s} < +\infty \} 
$$
with the norm given by $\| \Fb \|_{H^s}^2 := \sum_{n=-\infty}^{\infty} (1+|n|^{2s}) |\widehat{\Fb}_n|^2_F$, and we set $H^\infty := \bigcap_{s \geq 0} H^s(\T)$. We also define the family of Hardy--Sobolev spaces denoted by $H^s_+ := H^s \cap L^2_+$ for $s \geq 0$ together with the obvious definition of the space $H^\infty_+$.


For integers $d \geq 2$ and $0 \leq k \leq d$, we recall that the {\em complex Grassmannian} $\Gr_k(\C^d)$ are defined as the set of matrices such that
$$
\Gr_k(\C^d) := \left \{ U \in \C^{d \times d} : U = U^*, \; U^2 =\mathds{1}_d, \; \tr(U) = d-2k \right \},
$$
where $\mathds{1}_d$ denote the $d\times d$--identity matrix. We remind the reader that the trivial cases $\Gr_0(\C^d) = \{ \mathds{1}_d \}$ and $\Gr_d(\C^d) = \{ -\mathds{1}_d \}$ will be mostly excluded in our discussion, i.e., we assume that $1 \leq k \leq d-1$ holds. For real $s \geq 0$, we introduce the space
$$
H^{s}(\T; \Gr_k(\C^d)) := \left \{ \Ub \in H^{s}(\T; \C^{d \times d})) : \mbox{$\Ub(\theta) \in \Gr_k(\C^d)$ for a.e.~$\theta \in \T$} \right \}
$$
with the obvious definition $H^\infty(\T; \Gr_k(\C^d)) := \bigcap_{s \geq 0} H^s(\T; \Gr_k(\C^d))$. Furthermore, we recall that
$$
\mathcal{R}at(\T; \Gr_k(\C^d)) \subset H^\infty(\T; \Gr_k(\C^d))
$$ 
denotes the set of rational maps $\Ub : \T \to \Gr_k(\C^d)$, i.e., each matrix entry of $\Ub$ is a rational function of $z = \eu^{i \theta} \in \pt \D \cong \T$. Since $\Gr_k(\C^d)$ is a smooth compact manifold without boundary, the seminal work of Brezis--Nirenberg \cite{BrNi-95} tells us that the space of smooth maps $C^\infty(\T; \Gr_k(\C^d))$ is dense in $H^{1/2}(\T; \Gr_k(\C^d))$. In fact, it was shown in \cite{GeLe-25}[Theorem A.2], that we can restrict to rational maps to already obtain a dense subset. 

\begin{thm}[Density of rational maps] \label{thm:dense}
For any $d \geq 2$ and $0 \leq k \leq d$, the set $\mathcal{R}at(\T; \Gr_k(\C^d))$ is dense in $H^{1/2}(\T; \Gr_k(\C^d))$.
\end{thm}

\section{Spectral analysis of $T_\Ub$}

\label{sec:spec}

In this section, we derive some fundamental spectral properties of the Toeplitz operator
$$
T_\Ub : L^2_+(\T, \Cd) \to L^2_+(\T; \Cd), \quad \Fb \mapsto T_\Ub(\Fb) = \Pi(\Ub \Fb),
$$ 
for a given initial datum 
$$
\Ub \in H^{1/2}(\T; \Gr_k(\C^d)).
$$
Here and throughout the following, we take $d \geq 2$ and $1 \leq k \leq d-1$ to be fixed integers. Since $\Ub^* = \Ub$ and $\Ub^2 = \mathds{1}_d$ holds a.e.~on $\T$ holds, we immediately see that $T_{\Ub}^* = T_{\Ub}$ is self-adjoint and that $T_\Ub$ is a contraction, i.e., we have $\| T_{\Ub} \| \leq \| \Ub \|_{L^\infty} = 1$.  In fact, we will see below that equality $\|T_\Ub \| = 1$ holds.

\subsection{Key spectral properties}

Our starting point is the following key identity for the Toeplitz operator $T_\Ub$, where we recall that 
$$
H_\Ub : L^2_+(\T, \Cd) \to L^2_-(\T, \Cd), \quad \Fb \mapsto H_\Ub(\Fb) = \Pi_-(\Ub \Fb)
$$
denotes the Hankel operator with the matrix-valued symbol $\Ub \in H^{1/2}(\T; \Gr_k(\C^d))$. As usual, we often use the shorthand notation $L^2_+ = L^2_+(\T; \Cd)$ in what follows.

\begin{lem}[Key identity] \label{lem:key_identity}
We have the identity
$$
T_{\Ub}^2 = \id - K_{\Ub} \quad \mbox{with $K_{\Ub} = H_{\Ub}^* H_{\Ub}$}.
$$
Here the operator $K_\Ub : L^2_+ \to L^2_+$ is self-adjoint and trace-class with
$$
\Tr (K_{\Ub}) = \Tr ( H_{\Ub}^* H_\Ub)= \frac{1}{2} \| \Ub \|_{\dot{H}^{1/2}}^2.
$$
\end{lem}

\begin{remark*}
From the identity $T_{\Ub}^2 = \id - K_\Ub$ and the compactness of $K_\Ub$, we see that $T_\Ub$ is Fredholm with $\mathrm{ind} (T_\Ub) = 0$ by the self-adjointness of $T_\Ub$. Furthermore, we see that $0 \in \sigma(K_\Ub)$ by compactness of $K_\Ub$, whence it follows that equality $\| T_\Ub \| = 1$ holds.
\end{remark*}

\begin{proof}
 The claimed identity follows directly from \eqref{eq:ToeplitzHankel} using that $\Ub^2 = \mathds{1}_d$ and $\Ub= \Ub^*$ almost everywhere on $\T$. 
  
 To show that the self-adjoint operator $K_\Ub=H^*_\Ub H_\Ub$ is trace-class, we define $\mathbf{E}_{n,l} = \eu^{i n \theta} \mathbf{E}_l$ with $n \geq 0$, where the matrices $\mathbf{E}_l$  with $l=1, \ldots, d^2$ form an orthonormal basis of $\C^{d \times d}$. Thus the family $(\mathbf{E}_{n,l})_{n \geq 0, 1 \leq l \leq d^2}$ forms an orthonormal basis of $L^2_+(\T; \C^{d \times d})$. We calculate 
\begin{align*}
\Tr(H_{\Ub}^* H_\Ub ) & = \sum_{n=0}^\infty \sum_{l=1}^{d^2} \langle H_\Ub \mathbf{E}_{n,l}| H_\Ub \mathbf{E}_{n,l}  \rangle = \sum_{n=0}^\infty \sum_{l=1}^{d^2} \langle \Pi_{-}(\Ub \mathbf{E}_{n,l}) | \Pi_-(\Ub \mathbf{E}_{n,l}) \rangle \\
& = \sum_{n=0}^\infty \sum_{l=1}^{d^2} \left \langle \sum_{k+n <0} \widehat{\Ub}_k \eu^{i (k+n)\theta} \mathbf{E}_l | \sum_{m + n <0} \widehat{\Ub}_m \eu^{i (m+n) \theta} \mathbf{E}_l \right \rangle \\
& = \sum_{n=0}^\infty \sum_{n+k <0} |\widehat{\Ub}_k|_E^2 = \sum_{k<0}^\infty |k| |\widehat{\Ub}_k|_E^2 =  \| \Pi_- \Ub \|_{\dot{H}^{1/2}}^2 = \frac{1}{2} \| \Ub \|_{\dot{H}^{1/2}}^2.
\end{align*}
In the last step, we used that $\Ub = \Ub^*$ a.e.~on $\T$, which is easily seen to imply that $(\widehat{\Ub}_{k})^* = \widehat{\Ub}_{-k}$ for $k \in \Z$ and thus $\| \Pi_- \Ub \|_{\dot{H}^{1/2}}^2 = \| \Pi_+ \Ub \|_{\dot{H}^{1/2}}^2 = \frac{1}{2} \| \Ub \|_{\dot{H}^{1/2}}^2$. 
\end{proof}

Let $\sigma(T_\Ub)$ denote the spectrum of $T_\Ub$. We recall that the {\em discrete spectrum} $\sd(T_\Ub)$ and {\em essential spectrum} $\se(T_\Ub)$ are given by
$$
\sd(T_{\Ub}) = \{ \mu \in \sigma(T_\Ub) : \mbox{$\mu$ is an isolated eigenvalue with finite multiplicity} \},
$$
$$
\se(T_\Ub) = \{ \mu \in \sigma(T_\Ub) : \mbox{$T_{\Ub}-\mu \id$ is not Fredholm} \}.
$$
Note that we do not need to distinguish here between geometric and algebraic multiplicity because $T_\Ub=T_{\Ub}^*$ is self-adjoint. Also, by self-adjointness of $T_\Ub$, we have that
$\sd(T_\Ub) = \sigma(T_\Ub) \setminus \se(T_\Ub)$.

\begin{lem} \label{lem:Toeplitz_spectrum}
We have the following properties.
\begin{enumerate}
\item[(i)] The discrete spectrum of $T_\Ub$ is given by 
$$
\sd(T_{\Ub}) = \{ \mu \in \R : \mbox{$\lambda = 1-\mu^2 > 0$ is an eigenvalue of $K_\Ub$} \} 
$$ 
and $\sd(T_\Ub) \subset (-1,+1)$ is at most countable.

\item[(ii)] The essential spectrum of $T_\Ub$ is given by
$$
\se(T_\Ub) = \{ \pm 1 \}.
$$

\item[(iii)] We have the spectral representation
$$
T_{\Ub} = \sum_{\mu \in \sigma(T_\Ub)} \mu P_\mu,
$$
where $P_\mu$ denotes the orthogonal projection onto $\ker (T_\Ub - \mu \id)$. As a consequence, the spectrum of $T_{\Ub}$ is pure point.
\end{enumerate}
\end{lem}

\begin{proof}
Item (i) directly follows from the identity $T_\Ub^2 = \id - K_\Ub$ in Lemma \ref{lem:key_identity} for the self-adjoint operator $T_\Ub$ and by the compactness of the operator $K_\Ub$.

As for (ii), we note that $\se(T_\Ub) \subseteq \{ \pm 1 \}$ holds, since $K_\Ub= \id - T_\Ub^2$ is compact. To show the equality $\se(T_\Ub) = \{ \pm 1 \}$, we argue as follows. Since $1 \leq k \leq d-1$, we note that $-1$ and $+1$ are both eigenvalues of the matrix $\Ub(\theta) \in \Gr_k(\C^d)$ for a.e.~$\T$. Hence the matrix-valued map $\Gb_\sigma = \Ub + \mu \mathds{1}_d \in L^\infty(\T; \C^{d \times d})$ is not invertible in $L^\infty(\T; \C^{d \times d})$ for $\mu \in \{ \pm 1 \}$. From \cite{Pe-03}[Theorem 4.3, Chapter 3], we recall that $T_\Gb$ with $\Gb \in L^\infty(\T; \C^{d \times d})$ is Fredholm if and only if $\Gb^{-1} \in L^\infty(\T; \C^{d \times d})$. Hence we conclude that $T_{\Gb} = T_{\Ub + \mu \mathds{1}_d} = T_\Ub + \mu  \id$ is not Fredholm for $\mu \in \{ \pm 1 \}$. This shows that $\{ \pm 1 \} \subseteq \se(T_\Ub)$, which completes the proof of $\se(T_\Ub) = \{ \pm 1 \}$.

Finally, we remark that (iii) follows from the spectral theorem applied to the bounded self-adjoint operator $T_\Ub$.
\end{proof}

\subsection{On invariant subspaces}
As before, we assume that $\Ub \in H^{1/2}(\T; \Gr_k(\C^d))$ in what follows. By the commutator formula in Lemma \ref{lem:commutator}, we deduce that any eigenspace $\Es_\mu(T_\Ub)=\ker (T_\Ub - \mu \id)$ must be {\em nearly $S^*$-invariant}; see \cite{Sa-88}. That is, we have the implication:
$$
\mbox{$\Fb \in \Es_\mu(T_\Ub)$ and $\Mean(\Fb) = 0$} \quad \Rightarrow \quad S^* \Fb \in \Es_\mu(T_\Ub).
$$
This simple observation will turn out to be valuable in our analysis further below. 

Moreover, we can naturally identify a closed subspace in $L^2_+$ associated to $T_\Ub$ which is also $S^*$-invariant. Indeed, let us define the closed subspace
$$
\Hs := \ov{\ran(K_\Ub)} = \ov{\ran(H^*_\Ub H_\Ub)} = \ov{\ran(H^*_\Ub)}.
$$
This induces the orthogonal decomposition
$$
L^2_+ = \Hs \oplus \Hs^\perp \quad \mbox{with} \quad \mbox{$\Hs = \ov{\ran(H^*_\Ub)}$ and $\Hs^\perp = \ker(H_\Ub)$}.
$$
By a well-known result  (see, e.g.~\cite{Pe-03}), the kernel of a Hankel operator is invariant under the forward shift $S$. Therefore we have  $S (\Hs^\perp) \subset \Hs^\perp$, which implies that $S^*(\Hs) \subset \Hs$ by taking the adjoint. On the other hand, since $\Hs$ is the closed linear span of the eigenfunctions of $K_\Ub$ with positive eigenvalues $\lambda >0$, we easily infer that $T_\Ub(\Hs) \subset \Hs$ from Lemma \ref{lem:key_identity} above. Hence we obtain the following fact.

\begin{lem}
The closed subspace $\Hs = \ov{\ran(K_\Ub)}$ is invariant under $T_\Ub$ and $S^*$, i.e., we have $T_{\Ub}(\Hs) \subset \Hs$ and $S^*(\Hs) \subset \Hs$.
\end{lem}

\begin{remark*}
If $\Hs \subsetneq L^2_+$ is a proper subspace, which means that $\ker(H_\Ub) \neq \{ 0 \}$ is nontrivial, then by well-known {\em Beurling--Lax--Halmos theorem} it follows that $\Hs$ is a so-called {\em model space}; see e.g.~\cite{Ni-86}[Corollary I.6]. That is, we can write
$$
\Hs = (\Theta L^2_+(\T; V))^\perp
$$
with some subspace $V \subseteq \C^{d \times d}$ and some $\Theta \in L^\infty_+(\T; \Ls(V; \C^{d \times d}))$ such that 
$$
\Theta(\theta)^* \Theta(\theta) = \id_{V} \quad \mbox{for a.e.~$\theta \in \T$}.
$$ 
We say that $\Theta$ is a {\em left inner function}. In the case of equality $V = \C^{d \times d}$, we say that $\Theta$ is a {\em two-sided inner function}. However, we will not make further use of the this model space point of view in our analysis of \eqref{eq:HWM}.
\end{remark*}

With the help of the celebrated Kronecker theorem for Hankel operators, we can now easily identify the case when $\Hs$ happens to finite-dimensional. 

\begin{lem}[\`a la Kronecker] \label{lem:kronecker}
$\Ub \in \mathcal{R}at(\T; \Gr_k(\C^d))$ if and only if $\dim \Hs < \infty$.
\end{lem}

\begin{remark*}
In view of Lemma \ref{lem:Toeplitz_spectrum}, we notice that $\dim \Hs$ is finite-dimensional if and only if the discrete spectrum $\sd(T_\Ub)$ is finite. 
\end{remark*}

\begin{proof}
By Kronecker's theorem (see, e,\,g., \cite{Pe-03}[Chapter 2, Theorem 5.3], the  Hankel operator $H^*_{\Vb} : L^2_- \to L^2_+$ with $\Vb \in L^\infty(\T; \Cd))$ has finite rank if and only if its co-analytic part $\Pi_- \Vb(z) = \sum_{n<0} \widehat{\Vb}_n z^{-n}$ is a rational function in $z \in \D$.  On the other hand, since $\Ub = \Ub^*$ almost everywhere on $\T$, we deduce that 
$$
\Pi \Ub = (\Pi_-(\Ub^*)) ^*+ \Mean(\Ub) = (\Pi_- \Ub)^* + \Mean(\Ub)
$$ 
with the constant function $\Mean(\Ub) = \widehat{\Ub}_0$. This shows that $\Pi_+(\Ub)$ is rational  if and only if $\Pi_-(\Ub)$ is rational. Therefore, we have demonstrated that $\Ub = \Pi(\Ub)+ \Pi_-(\Ub)$ is rational if and only if $\Hs$ is finite-dimensional.
\end{proof}

\section{Lax pair structure}

\label{sec:lax}

We now elaborate on the Lax pair structure for \eqref{eq:HWM} posed on the torus $\T$. Informed by our previous work \cite{GeLe-25} for the half-wave maps equation posed on the real line, we introduce the following unbounded operator acting on $L^2_+$ with
\be
B_\Ub = -\frac{i}{2} ( T_\Ub  D + D  T_\Ub) + \frac{i}{2} T_{|D| \Ub}
\ee
where $D=-i \pt_\theta$ and $\Ub \in H^{s}(\T; \Gr_k(\C^d))$ with some $s> \frac{3}{2}$ is given. From the discussion in Appendix \ref{app:lwp}, we deduce that $B_\Ub$ with operator domain $\dom( B_\Ub) = H^1_+$ is {\em essentially skew-adjoint}.

For the rest of this subsection, we will always make the following assumptions without further reference:
\begin{itemize}
\item $I \subset \R$ is an interval with $0 \in I$.
\item $\Ub \in C(I; H^s(\T; \Gr_k(\C^d)))$ with some $s > \frac{3}{2}$ is a solution \eqref{eq:HWM}.
\end{itemize}
We refer the reader to Appendix \ref{app:lwp} below, where we prove local well-posedness of \eqref{eq:HWM} in $H^s$ with $s > 3/2$. We define the operators $T_{\Ub(t)}$ and $B_{\Ub(t)}$ for $t \in I$ in an obvious manner. Also, from the discussion in Appendix \ref{app:lwp}, we infer that there exists a unique solution $\Us(t) \in \Ls(L^2_+)$ for $t \in I$ solving the abstract operator-valued evolution equation
\be \label{eq:Us_ode}
\frac{d}{dt} \Us(t) = B_{\Ub(t)} \Us(t) \quad \mbox{for $t \in I$}, \quad \Us(0) = \id.
\ee
Since $B_{\Ub(t)} : H^1_+ \subset L^2_+ \to L^2_+$ is essentially skew-adjoint, the solution $\Us(t)$ is a unitary map on $L^2_+$ for all $t \in I$. Moreover, we have $\Us(t) : H^1_+ \to  H^1_+$ for $t \in I$ and $\Us(\cdot) \Fb \in C^1(I; L^2_+)$ for any given $\Fb \in H^1_+$. By the essential skew-adjointness of $B_{\Ub(t)}$, it also suffices consider elements in $H^1_+$ in the calculations below. In what follows, we will make use of these facts without further reference. 

We are now ready to state the following Lax pair property for sufficiently regular solutions of (HWM).

\begin{lem}[Lax pair structure] \label{lem:Lax}
We have the Lax evolution
$$
\frac{d}{dt} T_{\Ub(t)} = \left [B_{\Ub(t)}, T_{\Ub(t)} \right ] \quad \mbox{for} \quad  t \in I.
$$
As a consequence, we have the unitary equivalence
$$
T_{\Ub(t)} = \Us(t) T_{\Ub(0)} \Us(t)^* \quad \mbox{for} \quad t \in I.
$$
\end{lem}

\begin{remark*}
Note that, by the regularity imposed on $\Ub(t)$, it is easy to check that $T_{\Ub(t)}$ preserves the domain $\dom(B_{\Ub(t)})=H^1_+$. Hence the commutator above is well-defined on $\dom(B_{\Ub(t)})$ for any $t \in I$. Again, we will make tacitly use of such facts below.
\end{remark*}

\begin{proof}
Fix $t \in I$ and write $\Ub=\Ub(t)$ for notational simplicity. First, we  note $[D,T_\Ub] = T_{D \Ub}$ by the Leibniz rule, whence it follows
\begin{align*}
[B_{\Ub}, T_{\Ub}] & = -\frac{i}{2} [ T_\Ub  D + D  T_\Ub, T_\Ub] + \frac{i}{2} [T_{|D| \Ub}, T_\Ub] \\
& = -\frac{i}{2} \left (T_\Ub T_{D \Ub} + T_{D \Ub} T_\Ub \right ) + \frac{i}{2} [T_{|D| \Ub}, T_\Ub] =: (I) + (II).
\end{align*}
Differentiating the identity $\Ub^2 = \mathds{1}_d$ on $\T$, we find that $(D \Ub) \Ub + \Ub (D \Ub) = 0$. Thus $T_{(D \Ub) \Ub} + T_{(\Ub D) \Ub} = 0$ and from the general identity \eqref{eq:ToeplitzHankel} we get
$$
(I)  = -\frac{i}{2} \left ( T_{\Ub} T_{D \Ub} + T_{D \Ub} T_{\Ub} \right ) = \frac{i}{2}  H^*_\Ub H_{D \Ub} - \frac{i}{2} H^*_{D \Ub} H_\Ub  .
$$
Here we also used that $\Ub= \Ub^*$ and $(D \Ub)^* = -D \Ub$. Similarly, by invoking \eqref{eq:ToeplitzHankel} again and the facts that $H_{|D| \Ub} = -H_{D \Ub}$ and $H^*_{|D| \Ub} = -H^*_{D \Ub}$, we deduce
$$
(II)  = \frac{i}{2} [T_{|D|\Ub}, T_{\Ub}]  = -\frac{i}{2} T_{[\Ub, |D| \Ub]} - \frac{i}{2} H^*_\Ub H_{D \Ub} + \frac{i}{2} H^*_{D \Ub} H_{\Ub}.
$$
Adding the expression obtained for $(I)$ and $(II)$, wee that the parts containing the Hankel operators all cancel. Using that $\Ub= \Ub(t)$ solves \eqref{eq:HWM}, we find
$$
[B_\Ub, T_\Ub] = (I) + (II) = -\frac{i}{2} T_{[\Ub, |D| \Ub]} =  T_{\pt_t \Ub} = \frac{d}{dt} T_{\Ub}.
$$

The identity $T_{\Ub(t)} = \Us(t) T_{\Ub(0)} \Us(t)^*$ for $t \in I$ follows from the commutator identity and \eqref{eq:Us_ode}. We omit the details. This completes the proof of Lemma \ref{lem:Lax}.
\end{proof}

Recall that $K_{\Ub(t)} = \id - T_{\Ub(t)}^2$ by Lemma \ref{lem:key_identity} with the trace-class operator $K_{\Ub} = H^*_\Ub H_\Ub : L^2_+ \to L^2_+$. From Lemma \ref{lem:Lax} and the Leibniz rules for $\frac{d}{dt}$ and the commutator of operators, respectively, we directly deduce the following result.

\begin{cor} \label{cor:Lax2}
As a consequence of Lemma \ref{lem:Lax}, we have
$$
\frac{d}{dt} K_{\Ub(t)} = \left  [B_{\Ub(t)}, K_{\Ub(t)} \right ] \quad \mbox{for} \quad t \in I
$$
and the unitary equivalence
$$
K_{\Ub(t)} = \Us(t) K_{\Ub(0)} \Us(t)^* \quad \mbox{for} \quad t \in I.
$$
\end{cor}

For later use, let us also record another important feature of the Lax pair structure. To this end, we notice that we always have that $\C^{d \times d} \subset \ker(B_{\Ub(t)})$ holds, that is,
\be \label{eq:Bu_kernel}
B_{\Ub(t)} \Eb = 0 \quad \mbox{for any $\Eb \in  \C^{d \times d} \subset L^2_+$}.
\ee
Indeed, take a constant matrix $\Fb \in \C^{d \times d}$. Then, since $D \Eb= 0$ and $\Pi D = \Pi |D|$, we readily find
$$
B_{\Ub(t)} \Eb = -\frac{i}{2} D \Pi(\Ub(t) \Eb) + \frac{i}{2} \Pi (|D|\Ub(t) \Eb) = 0.
$$
Based on this simple observation, we can deduce that $\Pi \Ub(t)$ can be simply expressed in terms of the unitary map $\Us(t): L^2_+ \to L^2_+$ applied to the initial condition $\Pi \Ub(0)$. 

\begin{lem} \label{lem:HWM_Bu}
The following identity holds true:
$$
\Pi \Ub(t) = \Us(t) \Pi \Ub(0) \quad \mbox{for} \quad t \in I.
$$
\end{lem}

\begin{proof}
Since $\Us(t) : L^2_+ \to L^2_+$ is unitary, the claimed identity is equivalent to
$$
\Us(t)^* \Pi \Ub(t) = \Pi \Ub(0).
$$
Furthermore, from $\Us(0)^* = \id$, it suffices to show that the time derivative of left-hand side above is zero. Indeed, we find
\be \label{eq:Ub_evolv}
\frac{d}{dt} \left (\Us(t)^* (\Pi \Ub(t) \right ) = \Us(t)^* \left ( -B_{\Ub(t)} \Pi \Ub(t) + \pt_t \Pi \Ub(t) \right ),
\ee
using the skew-symmetry of $B_{\Ub(t)}$. Next, by the Lax equation in Lemma \ref{lem:Lax} and the fact $B_{\Ub(t)}(\mathds{1}_d)=0$ from above, we obtain
$$
\frac{d}{dt} T_{\Ub(t)}(\mathds{1}_d) = B_{\Ub(t)} T_{\Ub(t)}(\mathds{1}_d) - T_{\Ub(t)} B_{\Ub(t)}(\mathds{1}_d) = B_{\Ub(t)} \Pi(\Ub(t)).
$$
On the other hand, we evidently see that $\pt_t \Pi \Ub(t) = \frac{d}{dt} T_{\Ub(t)}(\mathds{1})$ and hence the right-hand side in \eqref{eq:Ub_evolv} is equal to zero. This completes the proof.
\end{proof}

\begin{remark*}
Since $\Us(t)$ solves \eqref{eq:Us_ode}, we deduce from Lemma \ref{lem:HWM_Bu} that
\be
\pt_t \Pi \Ub(t) = B_{\Ub(t)} \Pi \Ub(t) \quad \mbox{for $t \in I$}.
\ee
We remark that another (and more direct) way of obtaining this differential equation follows from applying the Cauchy--Szeg\H{o} projection $\Pi$ to \eqref{eq:HWM} itself together with the general identity
$$
\Pi (\Fb \Gb) = \Pi( \Fb \Pi(\Gb)) + \Pi(\Pi(\Gb) \Fb) - \Pi(\Fb) \Pi(\Gb)
$$
in combination with the pointwise identity $(D \Ub) \Ub + \Ub (D \Ub) = 0$ because of the pointwise constraint $\Ub^2 =\mathds{1}_d$. We leave the details to the interested reader.
\end{remark*}

\subsection{Explicit formula}

By further exploiting the Lax structure for \eqref{eq:HWM}, we shall now find an explicit formula for the unitary propagator $\Us(t)$ generated by $B_{\Ub(t)}$ in the setting of sufficiently regular solutions, that is, we always assume that $\Ub \in C(I; H^s(\T; \Gr_k(\C^d)))$ is a solution of \eqref{eq:HWM} with some $s > \frac{3}{2}$ on some  time interval $I \subset \R$ such that $0 \in I$. 

We start with some preliminaries, which will be essential for proving the explicit formula stated in Theorem \ref{thm:explicit_smooth} below. We begin by showing that the backshift operator $S^*$ enjoys the following unitary equivalence on the time interval $I$.

\begin{lem} \label{lem:S_star_unitary}
For any $t \in I$, it holds that
$$
\Us(t)^* S^* \Us(t) =  \eu^{-it T_{\Ub_0}} S^* \quad \mbox{on} \quad L^2_+.
$$
\end{lem}

\begin{proof}
We divide the proof into two steps as follows.

\medskip
\textbf{Step 1.}  Let $t \in I$ be given. We first prove the commutator identity
\be \label{eq:comm_SB}
 [S^*, B_{\Ub(t)}] \Fb = -\ii T_{\Ub(t)} S^* \Fb  \quad \mbox{for} \quad \Fb \in H^1_+.
\ee
Indeed, from the fact that $[S^*, D]  = S^*$ on $H^1_+$ and Lemma \ref{lem:commutator}, we find
\begin{align*}
[S^*, T_\Ub D + D T_\Ub] \Fb & = T_\Ub [S^*,D] \Fb + [S^*,T_\Ub] D \Fb + D [S^*, T_\Ub] \Fb + [S^*,D] T_\Ub \Fb \\
& = T_{\Ub} S^* \Fb + (S^* \Pi \Ub) \Mean(D \Fb) + D ( S^* \Pi \Ub) \Mean(\Fb) + S^* T_\Ub \Fb\\
& = T_\Ub S^* \Fb + S^* T_\Ub \Fb + D (S^* \Pi \Ub) \Mean(\Fb),
\end{align*}
where we used that $\Mean(D \Fb)=0$ for any $\Fb \in H^1_+$. Next, we notice
$$
[S^*, T_{|D| \Ub}] \Fb = (S^* \Pi |D| \Ub) \Mean(\Fb) = (S^* \Pi D \Ub) \Mean(\Fb) = (S^* D \Pi \Ub) \Mean(\Fb),
$$ 
since $\Pi |D| = \Pi D = D \Pi$. Recalling that $B_\Ub=-\frac{\ii}{2}(T_\Ub D + D T_\Ub) + \frac{\ii}{2} T_{|D|\Ub}$, we deduce
\begin{align*}
[S^*, B_\Ub] \Fb & = -\frac{\ii}{2} ( T_\Ub S^* \Fb + S^* T_\Ub \Fb ) + \frac{\ii}{2} [S^*,D](\Pi \Ub) \Mean(\Fb) \\
& = -\frac{\ii}{2} ( T_\Ub S^* \Fb + S^* T_\Ub \Fb ) + \frac{\ii}{2} (S^* \Pi \Ub) \Mean(\Fb) \\
& = -\ii T_{\Ub} S^* \Fb - \frac{\ii}{2} [S^*,T_\Ub] \Fb + \frac{\ii}{2} (S^* \Pi \Ub) \Mean(\Fb)  = -\ii T_{\Ub} S^* \Fb,
\end{align*}
using once again that $[S^*, T_\Ub] \Fb = (S^* \Pi \Ub) \Mean(\Fb)$. This proves \eqref{eq:comm_SB}.

\medskip
\textbf{Step 2.} Let $\Fb \in H^1_+$ be given. From \eqref{eq:comm_SB} we find
$$
\frac{d}{dt} \Us(t)^* S^* \Us(t) \Fb  = \Us(t)^* [S^*, B_{\Ub(t)}] \Us(t) \Fb = \Us(t)^* (-\ii T_{\Ub(t)} S^* ) \Us(t) \Fb.
$$
Since $T_{\Ub(t)} = \Us(t) T_{\Ub(0)} \Us(t)^*$ by the Lax evolution in Lemma \ref{lem:Lax} and the fact that $\Us(t)$ is unitary, we conclude
$$
\frac{d}{dt} \Us(t)^* S^* \Us(t) \Fb = -\ii T_{\Ub(0)} \Us(t)^* S^* \Us(t) \Fb.
$$
By integration in time and using that $\Us(0)= \id$, we infer 
$$
\Us(t)^* S^* \Us(t) \Fb = \eu^{-\ii t T_{\Ub_0}} S^* \Fb.
$$
By density, this identity extends to all of $\Fb \in L^2_+$. The proof of Lemma \ref{lem:S_star_unitary} is now complete.
\end{proof}

The next result shows that the mean is preserved by the unitary maps $\Us(t)$.

\begin{lem} \label{lem:Us_mean}
For any $\Fb \in L^2_+$ and $t \in I$, we have
$$
\Mean(\Us(t) \Fb) = \Mean(\Fb).
$$
\end{lem}

\begin{proof}
Let $\Fb \in H^1_+$ and $\Eb \in \C^{d \times d}$ be a constant matrix. Using the skew-symmetry of $B_{\Ub(t)}$, we readily deduce that 
$$
\frac{d}{dt} \langle \Us(t) \Fb | \Eb \rangle = \langle B_{\Ub(t)} \Ub(t) \Fb | \Eb \rangle = -\langle \Ub(t) | B_{\Ub(t)} \Eb \rangle = 0
$$
thanks to \eqref{eq:Bu_kernel}. This shows that $\frac{d}{dt} ( \Us(t) \Fb) \perp \Cd$ and hence $\frac{d}{dt} \Mean(\Us(t) \Fb)=0$, whence it follows that $\Mean(\Us(t) \Fb) = \Mean(\Fb)$ for all $\Fb \in H^1_+$. Since $\Mean$ and $\Us(t)$ are bounded operators on $L^2_+$, this identity extends to all $\Fb \in L^2_+$ by density.
\end{proof}

We are now ready to establish the explicit formula for \eqref{eq:HWM} in the setting of sufficiently smooth solutions on a given time interval $[0,T]$.

\begin{thm}[Explicit formula in $H^s$ with $s > \frac{3}{2}$] \label{thm:explicit_smooth}
Let $\Ub \in C(I; H^s(\T;\Gr_k(\C^d)))$ be a solution of \eqref{eq:HWM} with some $s > \frac{3}{2}$. Denote by $\{ \Us(t) \}_{t \in I}$ the  unitary maps on $L^2_+$ as given by \eqref{eq:Ub_evolv}. Then, for any $\Fb \in L^2_+$, it holds that
$$
(\Us(t) \Fb)(z) = \Mean \left ((\id- z \eu^{-it T_{\Ub_0}}S^*)^{-1} \Fb \right ) \quad \mbox{for $(t,z) \in I \times \D$}.
$$
In particular, we have that
$$
\Pi \Ub(t,z) = \Mean \left ( (\id - z \eu^{-it T_{\Ub_0}} S^*)^{-1} \Pi \Ub_0 \right ) \quad \mbox{for $(t,z) \in I \times \D$}.
$$
\end{thm}

\begin{proof}
Let $t \in I$ and $\Fb \in L^2_+$ be given. From the reproducing-kernel formula \eqref{eq:reproduce} we can write
$$
(\Us(t) \Fb)(z) = \Mean \left ( (\id - z S^*)^{-1} \Us(t) \Fb \right) \quad \mbox{for $z \in \D$}.
$$
Now, by Lemma \ref{lem:S_star_unitary}, we have $S^* \Us(t) = \Us(t) \eu^{-itT_{\Ub_0}} S^*$. If we take the geometric series expansion in $z \in \D$, we deduce the identity
$$
(\id-z S^*)^{-1} \Us(t) =  \Us(t) (\id-z \eu^{-it T_{\Ub_0}} S^*)^{-1} \quad \mbox{for $z \in \D$}.
$$
Therefore,
$$
(\Us(t) \Fb)(z) = \Mean \left ( \Us(t) (\id-z \eu^{-it T_{\Ub_0}} S^*)^{-1} \Fb \right ) = \Mean \left ( (\id-z \eu^{-it T_{\Ub_0}} S^*)^{-1} \Fb \right ),
$$
where the last equation follows from Lemma \ref{lem:Us_mean} above. 

Finally, let us take $\Fb = \Pi \Ub_0 \in L^2_+$ in the explicit formula derived above. In view of Lemma \ref{lem:HWM_Bu}, we conclude that
$$
\Pi \Ub(t,z) = (\Us(t) \Pi \Ub_0)(z) = \Mean \left ( ( \id - z \eu^{-it T_{\Ub_0}} S^*)^{-1} \Pi \Ub_0 \right ) 
$$ 
for any $t \in I$ and $z \in \D$. This completes the proof of Theorem \ref{thm:explicit_smooth}.
\end{proof}

\section{Global well-posedness for rational data}

\label{sec:rational}

In this section, we will use the explicit formula in Theorem \ref{thm:explicit_smooth} to show that rational initial data $\Ub_0 \in \mathcal{R}at(\T; \Gr_k(\C^d))$ will give rise to unique smooth global-in-time solutions of \eqref{eq:HWM}. As a by-product of our analysis, we prove that rational solutions are in fact quasiperiodic in time, which implies a-priori bounds on all Sobolev norms $\| \Ub(t) \|_{H^s}$ for any $s  >0$ in the rational case.

\subsection{Global existence for rational data}
Suppose  $\Ub_0 \in H^{1/2}(\T; \Gr_k(\C^d))$ is an initial datum for \eqref{eq:HWM} -- which is not necessarily rational for the moment. In view of the explicit formula stated in Theorem \ref{thm:explicit_smooth} valid for sufficiently smooth solutions, we now {\em define} the expression
\be \label{def:Us}
(\Us(t) \Fb)(z) := \Mean \left ( (\id-z \eu^{-it T_{\Ub_0}} S^*)^{-1} \Fb \right ) \quad \mbox{for} \quad (t,z)  \in \R \times \D
\ee
for any given $\Fb \in L^2_+=L^2_+(\T; \C^{d \times d})$. Of course, the map $\Us(t)$ depends on the initial datum $\Ub_0$. But for the sake of simplicity we have chosen to omit this dependence in our notation.

Let us begin with a spectral result, which rules out the existence of eigenvalues of the contraction $\eu^{-it T_{\Ub_0}} S^* : L^2_+ \to L^2_+$ on the unit circle. Recall that $\sp(T)$ denotes the point spectrum of an operator $T \in \Ls(H)$ on a Hilbert space $H$. Notice also here we do not assume that $\Ub_0$ is rational for the following result.

\begin{lem} \label{lem:sigma_p}
Let $\Ub_0 \in H^{1/2}(\T; \Gr_k(\C^d))$ be given. For all $t \in \R$, it holds that
$$
\sp ( \eu^{-i t T_{\Ub_0}} S^*) \cap \pt \D = \emptyset.
$$
\end{lem}

\begin{remark*}
For a sufficiently smooth solution $\Ub \in C(I; H^s)$ of \eqref{eq:HWM}, we can see that $\eu^{-it T_{\Ub_0}} S^*$ for $t \in I$ is unitarily equivalent to the backward shift $S^*$ thanks to Lemma \ref{lem:S_star_unitary}. Since it is well-known that $\sp(S^*) \cap \pt \D = \emptyset$, the assertion above immediately follows. However, the point here is that we do not assume the existence of a sufficiently smooth solution but we provide a direct proof of this claim. In fact, this result will turn out to be essential when extending rational solutions globally in time.
 \end{remark*}

\begin{proof}
We argue by contradiction. Assume there exists $\Fb \in L^2_+ \setminus \{ 0 \}$ such that
\be \label{eq:eigen_circle}
\eu^{-it T_{\Ub_0}} S^* \Fb = \lambda \Fb
\ee
for some $\lambda \in \pt \D$. Taking the $L^2$-norm on both sides while using that $\eu^{-it T_{\Ub_0}}$ is unitary and $|\lambda|=1$, we deduce that $\|S^* \Fb\|_{L^2} = \| \Fb \|_{L^2}$, which by \eqref{eq:Sdefect} implies
$$
\Mean(\Fb) = 0.
$$
From Lemma \ref{lem:commutator} this yields that $[S^*,T_{\Ub_0}] \Fb = 0$. Thus by applying $T_{\Ub_0}$ to \eqref{eq:eigen_circle} and using that $T_{\Ub_0}$ and $\eu^{-it T_{\Ub_0}}$ commute, we deduce 
$$
\eu^{-it T_{\Ub_0}} S^* T_{\Ub_0} \Fb = \lambda T_{\Ub_0} \Fb.
$$
Hence the eigenspace
$$
\Vs_\lambda := \ker ( \eu^{-it T_{\Ub_0}} S^* - \lambda \id) 
$$
satisfies $T_{\Ub_0}(\Vs_\lambda) \subset \Vs_\lambda$ and therefore $\eu^{it T_{\Ub_0}}(\Vs_\lambda) \subset \Vs_\lambda$ as well. Next, by going back to \eqref{eq:eigen_circle}, we infer that $S^* \Fb = \lambda \eu^{it T_{\Ub_0}} \Fb \in \Vs_\lambda$ showing that
$S^* \Fb \in \Vs_\lambda$. Hence $S^* \Fb$ is also an eigenfunction of $\eu^{-it T_{\Ub_0}} S^*$ with eigenvalue $\lambda$ and we deduce that $\Mean(S^* \Fb)=0$. By iteration, it follows that
$$
\Mean((S^*)^n \Fb) = \widehat{\Fb}_n = 0 \quad \mbox{for $n=0,1,2,\ldots$}
$$
But this shows that $\Fb=0$, which is the desired contradiction.
\end{proof}

As a next step, we introduce the closed subspace $\Ks$ in $L^2_+$ by setting
\be \label{def:Ks}
\Ks := \Hs + \C \Pi \Ub_0 + \C \mathds{1}_d \quad \mbox{with} \quad \Hs = \ov{\ran (H^*_{\Ub_0})}.
\ee
Recall that $\Hs$ is the closed subspace introduced in Section \ref{sec:spec} above, which was found to be invariant under both $T_{\Ub_0}$ and $S^*$.  In general,  it is easy to see that $\Pi \Ub_0 \not \in \Hs$. This fact  forces us to introduce the slightly larger subspace $\Ks$ above. 

\begin{prop} \label{prop:Ks}
It holds that 
$$T_{\Ub_0}(\Ks) \subset \Ks, \quad S^*(\Ks) \subset \Ks, \quad \mbox{and} \quad \Pi \Ub_0 \in \Ks.
$$ 
Moreover, we have $\dim \Ks < \infty$ if and only if $\Ub_0 \in \mathcal{R}at(\T; \Gr_k(\C^d))$.
\end{prop}

\begin{proof}
From \eqref{def:Ks} it is evidently true that $\Pi \Ub_0$ belongs to $\Ks$. Since $S^*(\Hs) \subset \Hs$ and $S^*(\mathds{1}_d) = 0$ and $S^* \Pi \Ub_0 = H^*_{\Ub_0}(\eu^{-i \theta}) \in \ran(H^*_{\Ub_0}) \subset \Ks$, we deduce that $\Ks$ is invariant under $S^*$. Furthermore, we note that $T_{\Ub_0}(\Hs) \subset \Hs$ and $T_{\Ub_0}(\mathds{1}_d) = \Pi \Ub_0 \in \Ks$ as well as
\begin{align*}
T_{\Ub_0}(\Pi \Ub_0) & = \Pi( \Ub_0 \Pi \Ub_0) = \Ub_0^2 - \Pi (\Ub_0( \id - \Pi) \Ub_0) \\
& = \mathds{1}_d - H^*_{\Ub_0}(\Pi_- \Ub_0) \in \C \mathds{1}_d + \Hs \subset \Ks.
\end{align*}
This proves that $T_{\Ub_0}(\Ks) \subset \Ks$. 

Finally, it is clear that $\dim \Ks < \infty$ if and only if $\dim \Hs < \infty$. But the latter statement is equivalent to $\Ub_0 \in \mathcal{R}at(\T; \Gr_k(\C^d))$ thanks to Lemma \ref{lem:kronecker}.
\end{proof}

By Proposition \ref{prop:Ks}, the subspace $\Ks$ is invariant under $\eu^{-it T_{\Ub_0}} S^*$. Because of $\Pi \Ub_0 \in \Ks$, it therefore suffices to consider the restrictions
$$
\eu^{-itT_{\Ub_0}} S^* |_{\Ks} \quad \mbox{and} \quad (\id - z \eu^{-it T_{\Ub_0}} S^*)^{-1} |_{\Ks} \quad \mbox{for $z \in \D$}
$$
in the expression for $\Us(t)$ introduced above when applied to $\Pi \Ub_0$. Moreover, for rational initial data $\Ub_0$, the subspace $\Ks$ is finite-dimensional and our analysis simplifies substantially, since the spectral result in Lemma \ref{lem:sigma_p} turns out to be sufficient for obtaining a-priori bounds on all $H^s$-norms uniformly in time. In fact, provided that $\Ub_0$ is rational, we can deduce that the map $t \mapsto \Us(t) \Ub_0$ is {\em quasi-periodic} in $t \in \R$ yielding a-priori bounds on any Sobolev norm. 

\begin{lem} \label{lem:quasi_periodic}
Let $\Ub_0 \in \mathcal{R}at(\T; \Gr_k(\C^d))$. Then the map $t \mapsto \Us(t) \Ub_0$ is \textbf{quasi-periodic} with respect to $t \in \R$ in the sense that
$$
(\Us(t) \Ub_0)(z) = \Gb(t \bm{\lambda}, z) \quad \mbox{for $(t, z) \in \R \times \D$},
$$ 
with some integer $N \geq 1$, some constant $\bm{\lambda}=(\lambda_1, \ldots, \lambda_N) \in \R^N$, and some bounded and continuous map $\Gb : \T^N \times \ov{\D} \to \C^{d \times d}$. In particular, we have the a-priori bounds
$$
\sup_{t \in \R} \| \Us(t) \Ub_0 \|_{H^s(\T)} \lesssim_{s,\Ub_0} 1 \quad \mbox{for any $s \geq 0$}.
$$
\end{lem}

\begin{remarks*}
1) The real numbers $\{ \lambda_k \}_{k =1}^N$ are the eigenvalues of the restriction of $T_{\Ub_0} |_\Ks$ onto the finite-dimensional invariant subspace $\Ks$.

2) Notice that the map $t \mapsto \Us(t) \Ub_0$ is {\em periodic} in $t \in \R$ if and only if all $\lambda_1, \ldots, \lambda_N$ lie on a line in $\mathbb{Q}$.
\end{remarks*}

\begin{proof}
Let $\Ub_0 \in \mathcal{R}at(\T; \Gr_k(\C^d))$ be given. Recall that $\dim \Ks < \infty$  by Proposition \ref{prop:Ks}. Since $T_{\Ub_0}(\Ks) \subset \Ks$ and by self-adjointness of $T_{\Ub_0}$, we can write
$$
T_{\Ub_0} |_\Ks = \sum_{n=1}^N \lambda_n P_n  \quad \mbox{for some $1 \leq N \leq \dim \Ks$},
$$
where $\lambda_n \in \R$ are the eigenvalues of $T_{\Ub_0} |_\Ks$ and $P_n : \Ks \to \Ks$ denote the corresponding orthogonal projections onto the eigenspaces $E_n = \ker(T_{\Ub_0}|_\Ks - \lambda_n \id )$ which satisfy $E_n \perp E_m$ if $n \neq m$.

For $\bm{\omega}=(\omega_1, \ldots, \omega_N) \in \T^N$, let us introduce the linear map 
$$
 \Om(\bm{\omega}) := \sum_{n=1}^N \eu^{-i \omega_n} P_n : \Ks \to \Ks.
$$
Note that $\Om(\bm{\omega})$ is unitary with $\Om(\bm{\omega})^* = \Om(-\bm{\omega})$ and $\Omega(\bm{\omega})$ commutes with $T_{\Ub_0}|_{\Ks}$. By a straightforward adaption of the proof of Lemma \ref{lem:sigma_p} and noticing that $S^*(\Ks) \subset \Ks$, we infer that
$$
\sigma( \Om(\bm{\omega}) S^* |_{\Ks}) \cap \pt \D = \emptyset \quad \mbox{for $\bm{\omega} \in \T^N$}.
$$
Note also that $\sigma( \Om(\bm{\omega}) S^* |_{\Ks})=\sp( \Om(\bm{\omega}) S^* |_{\Ks})$ since $\Ks$ is finite-dimensional. By compactness of $\T^N$, this implies that there is some constant $r < 1$ such that
\be \label{eq:sigma_omega}
\sigma(\Om(\bm{\omega}) S^* |_{\Ks} ) \subseteq \ov{\D}_r \quad \mbox{for $\bm{\omega} \in \T^N$},
\ee
where $\ov{\D}_r = \{ z \in \C : |z| \leq r \}$ is the closed disc with radius $r$. Next, we define the map $\Gb : \T^N \times \D \to \C^{d \times d}$ with
$$
\Gb(\bm{\omega}, z) := \Mean \left ( ( \id - z \Om(\bm{\omega}) S^* |_{\Ks})^{-1} \Pi \Ub_0 \right ).
$$
From \eqref{eq:sigma_omega} we deduce that, for each $\bm{\omega} \in \T^N$, the map $z \mapsto \Gb(\bm{\omega}, z)$ is a rational function in $z$ whose poles belong to the set $\{ z \in \C : z^{-1} \in \C \setminus \ov{\D}_r \} \subset \{ |z| \geq r^{-1} \}$. Thus $z \mapsto \Gb(\bm{\omega}, z)$ extends smoothly to the boundary $\pt \D$, which shows that $\Gb(\bm{\omega}, \cdot) \in H^s(\T; \C^{d \times d})$ for any $s \geq 0$. Since $\T^N$ is compact and by continuity of the $H^s$-norm of $\Gb(\bm{\omega}, \cdot)$ with respect to $\bm{\omega}$, we obtain
\be \label{ineq:apriori_quasi}
\sup_{\bm{\omega} \in \T^N} \| \Gb(\bm{\omega}, \cdot) \|_{H^s(\T)} \lesssim_{s,\Ub_0} 1 \quad \mbox{for any $s \geq 0$}.
\ee

Finally, we observe that 
$$
\eu^{-it T_{\Ub_0}} |_\Ks = \Omega(t \bm{\lambda})
$$
with the vector $\bm{\lambda}=(\lambda_1, \ldots, \lambda_n) \in \R^N$ containing the eigenvalues of $T_{\Ub_0} |_\Ks$ introduced above. Therefore,
$$
(\Us(t)\Ub_0)(z) = \Mean \left ( (\id-z \Omega(t \bm{\lambda}) S^*)^{-1} \Pi \Ub_0 \right ) = \Gb(t \bm{\lambda}, z) \quad \mbox{for $(t,z) \in \R \times \D$}.
$$
In view of the previous discussion of $\Gb$, we complete the proof of Lemma \ref{lem:quasi_periodic}.
\end{proof}

We are now ready to derive the main result of this section dealing with rational initial data for \eqref{eq:HWM}.

\begin{cor}[GWP for rational initial data] \label{cor:GWP_rational}
For every $\Ub_0 \in \mathcal{R}at(\T; \Gr_k(\C^d))$, there exists a unique global-in-time solution 
$$
\Ub \in C^\infty(\R \times \T)
$$
 of \eqref{eq:HWM} with initial datum $\Ub(0) = \Ub_0$. In addition, the following properties hold for all $t \in \R$:
\begin{itemize}
\item[(i)] \underline{A-priori bound on $H^s$}: It holds that
$$
\sup_{t \in \R} \| \Ub(t) \|_{H^s}  \lesssim_{s, \Ub_0} 1 \quad \mbox{for all $s \geq 0$}.
$$
\item[(ii)] \underline{Lax evolution}: There exists a unitary map $\Us(t)$ on $L^2_+$ such that
$$
T_{\Ub(t)} = \Us(t) T_{\Ub_0} \Us(t)^*,
$$
and we have the conservation laws
$$
\Mean(\Ub(t)) = \Mean(\Ub_0), \quad E[\Ub(t)] = E[\Ub_0] \quad \mbox{for all $t \in \R$}
$$
\item[(iii)] \underline{Explicit formula}: For any $z \in \D$ and $t \in \R$, it holds that
$$
\Pi \Ub(t,z) = (\Us(t)\Pi \Ub_0)(z) = \Mean \left ( (\id- z \eu^{-it T_{\Ub_0}}S^*)^{-1} \Pi \Ub_0 \right ).
$$
\item[(iv)] \underline{Preservation of rationality}: We have $\Ub(t) \in \mathcal{R}at(\T; \Gr_k(\C^d))$ for all $t \in \R$.
\end{itemize}
\end{cor}

\begin{proof}
Suppose that $\Ub_0 \in \mathcal{R}at(\T; \Gr_k(\C^d))$ let $\Ub \in C(I; H^s(\T;\Gr_k(\C^d)))$ denote be the corresponding maximal  solution given by the local well-posedness result in Lemma \ref{lem:lwp} for initial data in $H^s$ with $s>\frac{3}{2}$. From Theorem \ref{thm:explicit_smooth} we infer that $\Pi \Ub(t,z) = \Phi(t)(\Ub_0)(z)$ is given by the explicit formula for $t \in I$ and $z \in \D$. But in view of Lemma \ref{lem:quasi_periodic} we deduce that
$$
\sup_{t \in I} \| \Pi \Ub(t) \|_{H^s(\T)} < +\infty,
$$
which implies that $I = \R$ must hold.

Finally, we remark that items (i)--(iii) directly follow from Lemma \ref{lem:Lax}, Corollary \ref{cor:Lax2}, Theorem \ref{thm:explicit_smooth}, and Lemma \ref{lem:quasi_periodic}. To see that (iv) holds true, we notice that $\rank(K_{\Ub(t)}) = \rank(K_{\Ub_0})$ by the Lax evolution in (i) and we have $\rank(K_\Ub) < \infty$ if and only if $\Ub \in \mathcal{R}at(\T; \Gr_k(\C^d))$ in view of Lemma \ref{lem:kronecker}.  
\end{proof}

\section{Stability principle for explicit formulae}

In this section, we study the explicit formula from more general operator theoretic perspective. Our key result will be a general {\em stability principle} formulated in Theorem \ref{thm:stony} below, which can be applied to completely integrable PDEs in Hardy spaces $L^2_+(\T;E)$ that feature explicit formulae such as the Benjamin--Ono equation, Calogero--Moser DNLS, and half-wave maps equation posed on $\T$. In our companion work, we develop a corresponding stability principle for explicit formulae involving the Hardy spaces $L^2_+(\R;E)$, which arises for the corresponding completely integrable PDEs on the real line.

The main ideas developed in this section will show a close connection to the celebrated {\em Wold decomposition} for isometries on Hilbert spaces; see e.g.~\cite{SzFoBeKe-10, Ni-86}. However, our presentation will be self-contained. Thus it is (hopefully) accessible to readers without any further operator-theoretic background. In particular, we will not take the Wold decomposition as given, but we can rather derive this result as a corollary below.

\subsection{Main result of this section}
The goal of this subsection is to derive a general result for explicit formulae that arise for the half-wave maps equation on $\T$ as well as the Benjamin--Ono equation and the Calogero--Moser DNLS on $\T$, together with their corresponding zero-dispersion limits. 

We consider the following general setting. Assume that $E$ is a complex Hilbert space and let $H=L^2_+(\T;E)$ be the corresponding $L^2$-based Hardy space of holomorphic functions on $\D$ with values in $E$. As usual, we use $S$ and $S^*$ to denote the forward and backward shift on $H$, respectively.

Next, we suppose that $L : \dom(L) \subset H \to H$ is a (possibly unbounded) self-adjoint operator. We let $\{ \eu^{-it L} \}_{t \in \R}$ denote the strongly continuous one-parameter unitary group on $H$ generated by $L$.  Finally, for $t \in \R$, we define the linear map
$$
\Us(t) : L^2_+(\T;E) \to L^2_+(\T;E)
$$
given by
$$
(\Us(t) \Fb)(z) = \Mean \left ( (\id-z \eu^{-itL} S^*)^{-1} \Fb \right )
$$
with $\Fb \in L^2_+(\T;E)$ and $z \in \D$. Recall that $\Mean(\Gb) = \frac{1}{2 \pi} \int_0^{2 \pi} \Gb  = \widehat{\Gb}_0$ denotes the mean, which corresponds to the orthogonal projection in $L^2_+(\T;E)$ onto the subspace constant functions on $\T$ valued in $E$. We have the following examples.

\begin{itemize}
\item (HWM) on $\T$: $E=\C^{d \times d}$ and $L=T_{\Ub_0}$.
\item (CS-DNLS) on $\T$: $E= \C$ and $L=D \pm T_{u_0} T_{u_0}$.
\item (BO) on $\T$: $E= \C$ and $L=D - T_{u_0}$. 
\item Zero-dispersion limit of (BO) on $\T$: $E=\C$ and $L=T_{u_0}$.
\end{itemize}

By the general discussion in Subsection \ref{subsec:prelim} below, it readily follows that $\Us(t)$ is always a contraction, i.e., its operator norm is $\| \Us(t) \| \leq 1$ for all $t\in \R$. The following general result now completely identifies the case when $\Us(t)$ is {\em unitary}. This insight will turn out to be essential for our analysis of (HWM) with data in $H^{1/2}$.
\begin{thm}[Stability principle] \label{thm:stony}
The map $\Us(t) : L^2_+(\T;E) \to L^2_+(\T;E)$ is unitary for all $t \in \R$ if and only if one of the following equivalent conditions hold.
\begin{enumerate}
\item[(i)] $\ker \, \Us(t) = \{ 0 \}$ for all $t \in \R$.
\item[(ii)] $\displaystyle \lim_{j \to \infty} \| (\eu^{-it L} S^*)^j \Fb \| = 0$ for all $\Fb \in L^2_+(\T;E)$ and $t \in \R$.
\end{enumerate}
In this case, the map $t \mapsto \Us(t) \Fb$ belongs to $C(\R; L^2_+(\T;E))$ for any $\Fb \in L^2_+(\T;E)$.
\end{thm}

Before we turn to the proof of this theorem, we will first derive some preliminary results in the next subsection.

\subsection{Preliminaries for the proof}
\label{subsec:prelim}
Let $H$ be a complex Hilbert space (not necessarily separable) with inner product and norm denoted by $\langle \cdot | \cdot \rangle$ and $\| \cdot \|$, respectively. In the following discussion, we always assume that
$$
\mbox{$\Sigma : H \to H$ is an {\em isometry}},
$$
which by definition means that $\| \Sigma \Fb \| = \| \Fb \|$ for each $\Fb \in H$. In fact, the latter condition is easily seen to be equivalent to
\be \label{eq:Sigma_isom}
\Sigma^* \Sigma = \id,
\ee
where $\id$ denotes the identity on $H$. Let us define $P_\Sigma \in \Ls(H)$ by setting
$$
P_\Sigma := \id - \Sigma \Sigma^*.
$$
In view of \eqref{eq:Sigma_isom}, we check that $P_\Sigma^* = P_\Sigma = P_\Sigma^2$ is an orthogonal projection on $H$ onto $E:= \ker \, \Sigma^*$, that is, we have
$$
E = \ker \, \Sigma^* = \ran \, P \quad \mbox{and} \quad E^\perp = \ran \, \Sigma = \ker \, P.
$$
Since $E \subset H$ is a closed subspace, it is a Hilbert space in its own right with norm $\| \Fb \|_E = \| \Fb \|$ for $\Fb \in E$. Let $L^2_+(\T;E)$ denote the corresponding Hardy space of $E$-valued holomorphic functions $\Gb: \D \to E$ such that $\Gb= \sum_{n=0}^\infty \Gb_n z^n$ with $\Gb_n \in E$ and $\sum_{n=0}^\infty \| \Gb_n \|_E^2 < \infty$. To simplify our notation, we will sometimes write $L^2_+$ instead of $L^2_+(\T;E)$.

For the given isometry $\Sigma : H \to H$, let us now define the linear map $\Us_\Sigma : H \to L^2_+(\T,E)$ by setting
\be
\boxed{ (\Us_\Sigma \Fb )(z) := P_\Sigma \left ( (\id- z \Sigma^*)^{-1} \Fb \right )}
\ee
with $\Fb \in H$ and $z \in \D$. Since $\| \Sigma^* \|  = \| \Sigma \|=1$ and thus $\sigma(\Sigma^*) \subseteq \ov{\D}$, we notice that $\id-z \Sigma^*$ has a bounded inverse for all $z \in \D$. In fact, we have the following result.

\begin{lem} \label{lem:UsSigma}
The linear map $\Us_\Sigma : H \to L^2_+(\T;E)$ is a contraction. For all $\Fb \in H$, the following properties hold.
\begin{enumerate}
\item[(i)] We have the identity
$$
\| \Us_\Sigma \Fb \|_{L^2_+}^2 = \| \Fb \|^2 - \lim_{N \to \infty} \| (\Sigma^*)^N \Fb \|^2.
$$
\item[(ii)] $\Us_\Sigma(\Fb) = \Fb$ whenever $\Fb \in E$.
\item[(iii)] The following intertwining relations hold
$$
S \Us_\Sigma = \Us_\Sigma \Sigma \quad \mbox{and} \quad S^* \Us_\Sigma = \Us_{\Sigma} \Sigma^* \quad \mbox{on $H$}.
$$ 
\end{enumerate}
\end{lem}

\begin{remark*}
Since $\| \Sigma^* \| = \| \Sigma \| = 1$, the non-negative sequence $\| \Fb \|_{L^2} \geq \| \Sigma^* \Fb \|_{L^2} \geq \| (\Sigma^*)^2 \Fb \|_{L^2} \geq \ldots \geq 0$ is monotone-decreasing. Hence the limit 
$$
\lim_{N \to \infty} \| (\Sigma^*)^N \Fb \| = \inf_{N \geq 0} \| (\Sigma^*)^N \Fb \|
$$ 
always exists for any $\Fb \in H$.
\end{remark*}

\begin{proof}
Let $\Fb \in H$ be fixed.  For any $z \in \D$, we can use the geometric series expansion to conclude
\be \label{eq:UsSigma_expansion}
(\Us_\Sigma \Fb)(z) = \sum_{n=0}^\infty (P_{\Sigma} (\Sigma^*)^n \Fb) z^n.
\ee
Note that the series is locally uniformly convergent in $\D$ because $\| P_\Sigma (\Sigma)^* \Fb \|_E \leq \| (\Sigma^*)^n \Fb \| \leq \| \Fb \|$ independent of $n$. Thus $\Us_\Sigma \Fb : \D \to E$ is holomorphic on $\D$ and we have $\Us_\Sigma \Fb \in L^2_+(\T,E)$ if and only if
$$
\| \Us_\Sigma \Fb \|_{L^2_+}^2 = \sum_{n=0}^\infty \| P_\Sigma (\Sigma^*)^n \Fb \|_E^2 < \infty.
$$
To see that the series is finite, we recall that $P_\Sigma = \id - \Sigma \Sigma^*$, which implies that
$$
\| P_\Sigma \Gb \|_E^2 = \| P_\Sigma \Gb \|^2 = \| \Gb \|^2 - \| \Sigma^* \Gb \|^2 \quad \mbox{for all $\Gb \in H$}.
$$
Applying this identity to $\Gb = (\Sigma^*)^n \Fb$ with $n \in \N$, we find
\be \label{eq:PSigma_1}
\| P_\Sigma (\Sigma^*)^n \Fb \|_E^2 = \| (\Sigma^*)^n \Fb \|^2  - \| (\Sigma^*)^{n+1} \Fb \|^2. 
\ee
Therefore, we obtain the telescopic sum
$$
\sum_{n=0}^{N-1} \| P_\Sigma (\Sigma^*)^n \Fb \|_E^2 = \| \Fb \|^2 - \|(\Sigma^*)^N \Fb \|^2 \quad \mbox{for $N \geq 1$}.
$$
This shows that $\| \Us_\Sigma \Fb \|_{L^2_+} \leq \| \Fb \|$ for all $\Fb \in H$. Hence the linear map $\Us_\Sigma : H \to L^2_+$ is well-defined and, in fact, it is a contraction. By passing to the limit $N \to \infty$, we deduce that
$$
\| \Us_\Sigma \Fb \|_{L^2_+}^2 = \| \Fb \|^2 - \lim_{N \to \infty} \| (\Sigma^*)^N \Fb \|^2.
$$ 
This proves item (i).

Next, we assume that $\Fb \in E = \ran \, P_\Sigma = \ker \, \Sigma^*$. Then \eqref{eq:UsSigma_expansion} immediately yields that $\Us_\Sigma \Fb = P_\Sigma(\Fb) = \Fb$ whenever $\Fb \in E$. This shows (ii). 

It remains to prove (iii). Recall that $S$ denotes the forward shift on $L^2_+(\T,E)$. For any $\Fb \in L^2_+(\T,E)$ and $z \in \D$, we thus find
\begin{align*}
S (\Us_\Sigma \Fb)(z) & = \sum_{n=0}^\infty P_\Sigma( (\Sigma^*)^n \Fb ) z^{n+1} = \sum_{n=1}^\infty P_\Sigma ( (\Sigma^*)^{n-1}  \Fb ) z^n \\
& = \sum_{n=1}^\infty P_\Sigma( (\Sigma^*)^n \Sigma \Fb ) z^n = \sum_{n=0}^\infty P_\Sigma((\Sigma^*)^n \Sigma \Fb) z^n = (\Us_\Sigma \Sigma \Fb)(z),
\end{align*}
where we used that $\Sigma^* \Sigma = \id$ (since $\Sigma$ is an isometry) and that $P_{\Sigma}( \Sigma \Fb) = 0$. Likewise, for the backward shift $S^*$ on $L^2_+(\T,E)$, we find
$$
S^* (\Us_\Sigma \Fb)(z) = \sum_{n=0}^\infty P_\Sigma((\Sigma^*)^{n+1} \Fb) z^n = \sum_{n=0}^\infty P_\Sigma ( (\Sigma^*)^n \Sigma^* \Fb) z^n =  \Us_\Om(\Sigma^* \Fb)(z).
$$
This shows (iii) and completes the proof of Lemma \ref{lem:UsSigma}.
\end{proof}

As a next step, we identify the case when $\Us_\Sigma$ happens to be a unitary map.

\begin{lem} \label{lem:UsSigma_unitary}
The map $\Us_\Sigma : H \to L^2_+(\T;E)$ is unitary if and only if one of the following equivalent conditions holds.
\begin{enumerate}
\item[(i)] $\ker \, \Us_\Sigma = \{ 0 \}$.
\item[(ii)] $\displaystyle \lim_{j \to \infty} \| (\Sigma^*)^j \Fb \| = 0$ for all $\Fb \in H$.
\end{enumerate}
In this case, we have the unitary equivalence
$$
\Us_\Sigma^* S \Us_\Sigma = \Sigma \quad \mbox{on $H$}.
$$
\end{lem}

\begin{proof} 
We divide the proof into the following steps.

\medskip
\textbf{Step 1.} We show that $\Us_\Sigma : H \to L^2_+(\T;E)$ is unitary if and only if (i) holds. Clearly, the unitarity of $\Us_\Sigma$ implies (i). On the other hand, let us assume that $\ker \, \Us_\Sigma = \{ 0 \}$ is trivial. Define the closed subspace
$$
\Xs := \ov{\mathrm{lin}} \, \{ \Sigma^j E \}_{j=0}^\infty.
$$
We claim that $\Us_\Sigma$ is an isometry on $\Xs$, i.e.,
\be \label{eq:UsSigma_isom}
\| \Us_\Sigma \Fb \|_{L^2_+} = \| \Fb \| \quad \mbox{for all $\Fb \in \Xs$}.
\ee
To see this, take any finite linear combination $\Fb = \sum_{j=0}^m \Sigma^j \Eb_j$ with $\Eb_j \in E$. Then $(\Sigma^*)^{m+1} \Fb = \sum_{j=0}^m (\Sigma^*)^{m+1-j} \Eb_j = 0$ using that $\Sigma^* \Sigma = \id$ and $(\Sigma^*)^k E = 0$ for $k \geq 1$. Thus we see that $\lim_{N \to \infty} \| (\Sigma^*)^N \Fb \| \leq \| (\Sigma^*)^{m+1} \Fb \| = 0$. From the identity in Lemma \ref{lem:UsSigma} we deduce that \eqref{eq:UsSigma_isom} holds on $\mathrm{lin} \, \{ \Sigma^j E \}_{j=0}^\infty$, which by density extends to its closure $\Xs$. This proves \eqref{eq:UsSigma_isom}.

Next, we claim that $\Us_\Sigma : \Xs \to L^2_+(\T;E)$ is surjective. Since isometries have closed range, it suffices to show $\Us_\Sigma(\Xs)$ is dense in $L^2_+(\T;E)$. By Lemma \ref{lem:UsSigma}, we have $E = \Us_\Sigma(E)$ and, by iterating the intertwining identity, we obtain $S^j E = \Us_\Sigma(\Sigma^j E)$ for $j \geq 1$. Thus we obtain that
$$
S^j E = z^j E = \Us_\Sigma(\Sigma^j E) \subset \Us_\Sigma(\Xs) \quad \mbox{for all $j \geq 0$}.
$$
Since the linear span of $\{ z^j E \}_{j=0}^\infty$ is dense in $L^2_+(\T;E)$, we conclude that $\Us(\Xs) = L^2_+(\T;E)$. In summary, we have shown that 
$$
\mbox{$\Us_{\Sigma} : \Xs \to L^2_+(\T;E)$ is unitary}.
$$

It remains to show that 
$$
H = \Xs.
$$ 
Since $\Xs$ is closed, this follows if we show that $\Xs^\perp = \{ 0 \}$. Indeed, let $\Fb \in \Xs^\perp$ be given. Then we necessarily have that $\Fb \perp \Sigma^j E$ for $j \geq 0$. Hence $\langle \Fb | \Sigma^j E \rangle = \langle (\Sigma^*)^j \Fb | E \rangle = 0$ for $j \geq 0$, which means that $(\Sigma^*)^j \Fb \in \ker \, P_\Sigma = E^\perp$ for $j \geq 0$. But in view of \eqref{eq:UsSigma_expansion} this implies that $\Us_\Sigma(\Fb)=0$, whence it follows that $\Xs^\perp$ belongs to $\ker \, \Us_\Sigma$. Since $\ker \, \Us_\Sigma = \{ 0 \}$ by assumption, this implies that $\Xs^\perp$ is trivial.

This completes the proof that condition (i) is  also sufficient for $\Us_\Sigma : H \to L^2_+(\T;E)$ to be unitary.

\medskip
\textbf{Step 2.} The implication (ii) $\Rightarrow$ (i) is clear from Lemma \ref{lem:UsSigma}. Assume now that (i) holds. Then $\Us_\Sigma : H \to L^2_+(\T;E)$ is unitary by Step 1 above. By the identity in Lemma \ref{lem:UsSigma}, we deduce that (ii) must hold.

\medskip
\textbf{Step 3.} If $\Us_\Sigma : H \to L^2_+(\T;E)$ is unitary, we directly deduce that $\Us_\Sigma^* S \Us_\Sigma = \Sigma$ on $H$ from the intertwining relation in Lemma \ref{lem:UsSigma}.

\medskip
The proof of Lemma \ref{lem:UsSigma_unitary} is now complete.
\end{proof}

\begin{remark*}
For readers with an operator-theoretic background, we remark that a straightforward generalization of the previous proof yields the well-known {\em Wold decomposition} for isometries on Hilbert spaces. More precisely, by a slight extension of the proof above, we can deduce the unique orthogonal decomposition
$$
H = \Xu \oplus \Xs
$$
with the following properties:
\begin{itemize}
\item $\Xu = \ker \, \Us_\Sigma = \{ \Fb \in H : \mbox{$\| (\Sigma^*)^j \Fb = \| \Fb \|$ for $j \geq 0$} \}$
\item $\Xs = \ov{\mathrm{lin}} \, \{ \Sigma^j E \}_{j=0}^\infty = \{ \Fb \in H : \lim_{j \to \infty} \| (\Sigma^*)^j \Fb \| = 0 \}$.
\item $\Us_\Sigma : \Xs \to L^2_+(\T;E)$ is unitary and $\Us_\Sigma^* S \Us_\Sigma = \Sigma$ on $\Xs$. 
\end{itemize}
We readily check that $\Xu$ and $\Xs$ are invariant under $\Sigma$ and $\Sigma^*$, where $\Sigma |_{\Xu}$ is unitary and $\Sigma |_{\Xs}$ is unitarily equivalent to right shift $S$ on $L^2_+(\T;E)$, i.e., $\Sigma |_{\Xs}$ is completely non-unitary (c.n.u.). This is the well-known Wold decomposition for the isometry $\Sigma$ in $H$. 

Alternatively, we could have invoked the Wold decomposition for $\Sigma$ and have worked `backwards' to deduce Lemma \ref{lem:UsSigma_unitary} above. But since the map $\Us_\Sigma$ itself is our central object of interest and to keep the discussion self-contained, we have chosen not to take the Wold decomposition as given. \end{remark*}

For later use, we also derive the following convergence result. Note that we impose strong operator convergence on the adjoints.

\begin{lem} \label{lem:Us_Sigma_conv}
Let $\Sigma$ and $\Sigma_n$ with $n \in \N$ be isometries in $H$ such that $\ker \, \Sigma^*_n = \ker \, \Sigma^* = E$ for $n \in \N$. If $\Sigma_n^* \to \Sigma^*$ strongly, then $\Us_{\Sigma_n} \to \Us_{\Sigma}$ weakly.
\end{lem}

\begin{proof}
Since $\ker \, \Sigma^*_n = \ker \, \Sigma^* = E$, we note that $P_{\Sigma_n} = P_{\Sigma}$ for all $n \in \N$. Let $\Fb \in H$ be given. We need to show that
\be \label{eq:U_Om_weak}
\Us_{\Sigma_n} \Fb \weakto \Us_\Sigma \Fb \quad  \mbox{in $L^2_+(\T;E)$}.
\ee
Define the sequence $\Gb_n := \Us_{\Sigma_n} \Fb$ in $L^2_+(\T;E)$. Note that $\| \Gb_n \|_{L^2_+} \leq \| \Fb \|$ because $\Us_{\Sigma_n} : H \to L^2_+(\T;E)$ is a contraction. Up to passing to a subsequence, we can assume that $\Gb_n \weakto \Gb$ in $L^2_+(\T;E)$ with some limit $\Gb \in L^2_+(\T;E)$. We claim that 
$$
\Gb(z) = (\Us_{\Sigma} \Fb)(z)  \quad \mbox{for $z \in \D$}.
$$ 
By holomorphicity on $\D$, it suffices to show this claim on the smaller disk $\D_{1/2} = \{ z \in \C : |z| < 1/2 \}$. Indeed, we observe
\begin{align*}
(\Us_{\Sigma_n} \Fb)(z) - (\Us_\Sigma \Fb)(z) & = P_{\Sigma} \left ( (\id-z \Sigma_n^*)^{-1} \Fb - (\id - z \Sigma^*)^{-1} \Fb \right ) \\
& = P_\Sigma \left ( (\id- z \Sigma_n^*)^{-1} (z(\Sigma^*- \Sigma_n^*)) (\id-z \Sigma^*)^{-1} \Fb \right ).
\end{align*}
Since $\sigma(\Sigma_n^*) \subseteq \ov{\D}$, we see that $\| (\id -z \Sigma_n^*)^{-1} \| \leq \frac{1}{1/2} = 2$ for $z \in \D_{1/2}$. Hence, for any $z \in \D_{1/2}$, we find
$$
\| (\Us_{\Sigma_n} \Fb)(z) - (\Us_{\Sigma} \Fb)(z)\|_E \leq  \| (\Sigma^*- \Sigma_n^*) (\id-z \Sigma^*)^{-1} \Fb \| \to 0 \quad \mbox{as $n \to \infty$},
$$
using that $\Sigma^*_n \to \Sigma^*$ strongly as operators on $H$. This proves that the weak limit is given by $\Gb = \Us_\Sigma \Fb$. Since the limit is independent of the chosen subsequence, we obtain that \eqref{eq:U_Om_weak} holds.
\end{proof}

\subsection{Proof of Theorem \ref{thm:stony}}

We consider the Hardy space $H = L^2_+(\T;E)$, where $E$ is some complex Hilbert space. For $t \in \R$, we define the isometry 
$$
\Sigma(t) := S \eu^{it L} : H \to H.
$$ 
Since $\Sigma(t)^* = \eu^{-itL} S^*$ and $\eu^{-it L}$ is unitary, we readily check that $\ker \, \Sigma(t)^* = \ker \, S^* = E$ and that $P_{\Sigma(t)} = \id - \Sigma(t) \Sigma(t)^* = \id - S S^* = \Mean$ for all $t \in \R$.

Hence we can invoke Lemma \ref{lem:UsSigma_unitary} to infer that $\Us(t) : L^2_+(\T;E) \to L^2_+(\T;E)$ is unitary if and only if one the equivalent conditions (i) and (ii) in Theorem \ref{thm:stony} holds.

Suppose now that $\Us(t) : L^2_+(\T;E) \to L^2_+(\T;E)$ is unitary for all $t \in \R$. Let $\Fb \in L^2_+(\T; E)$ be given and consider a sequence of times $t_n \to t$ with some time $t \in \R$. By the strong continuity of the unitary group $\{ \eu^{-it L} \}_{t \in \R}$, we deduce that 
$$
\Sigma(t_n)^* = \eu^{-it_n L} S^* \to  \eu^{-it L} S^*= \Sigma(t)^*
$$ 
strongly as operators. By Lemma \ref{lem:Us_Sigma_conv}, we have that $\Us(t_n) \Fb \weakto \Us(t) \Fb$ in $L^2_+(\T;E)$. Since the maps $\Us(t_n)$ and $\Us(t)$ are unitary, we deduce that $\lim_{n \to \infty} \| \Us(t_n) \Fb \| = \|\Us(t) \Fb \| = \| \Fb \|$, which implies that $\Us(t_n) \Fb \to \Us(t) \Fb$ in $L^2_+(\T;E)$. Hence the map $t \mapsto \Us(t) \Fb$ belongs to $C(\R; L^2_+(\T;E))$.

This proof of Theorem \ref{thm:stony} is now complete. \hfill $\Box$

\section{Global well-posedness in $H^{1/2}$}
\label{sec:gwp}

The goal of this section is to prove global well-posedness for (HWM) as stated in Theorem \ref{thm:gwp}. We organize the discussion as follows.

\subsection{Extension to initial data in $H^{1/2}$} 
Let us fix integers $d \geq 2$ and $1 \leq k \leq d-1$. Suppose that $\Ub_0 \in H^{1/2}(\T; \Gr_k(\C^d))$ is an initial datum for \eqref{eq:HWM}.  By Theorem \ref{thm:dense}, there is a sequence of rational initial data $\Ub_{0,n} \in \mathcal{R}at(\T; \Gr_k(\C^d))$ such that 
\be \label{eq:Urat_conv}
\Ub_{0,n} \to \Ub_0 \quad \mbox{in} \quad H^{1/2}.
\ee
Let $\Ub_n = \Ub_n(t) \in C(\R; H^\infty)$ denote the unique global-in-time solutions of (HWM) with initial datum $\Ub_n(0) = \Ub_{0,n}$ as provided by Corollary \ref{cor:GWP_rational} above. We use the short-hand notation $L^2_+ = L^2_+(\T; \C^{d \times d})$ in what follows. For $t \in \R$, we consider the maps $\Us_n(t)  : L^2_+ \to L^2_+$ and $\Us(t) : L^2_+ \to L^2_+$ with
$$
(\Us_n(t) \Fb)(z) = \Mean \left ( ( \id - z \eu^{-it T_{\Ub_0,n}} S^*)^{-1} \Fb \right ), 
$$
$$ 
(\Us(t) \Fb)(z) = \Mean \left ((\id - z \eu^{-it T_{\Ub_0}} S^*)^{-1} \Fb \right ).
$$
We record the following facts valid for any $t \in \R$, sequences $t_n \to t$, and $n \in \N$:
\begin{enumerate}
\item[(P$_1$)] $\Us_n(t) : L^2_+ \to L^2_+$ is unitary.
\item[(P$_2$)] $\Pi \Ub_n(t) = \Us_n(t) \Pi \Ub_{0,n}$.
\item[(P$_3$)] $\eu^{-i t_n T_{\Ub_0,n}} S^*  \to  \eu^{-it T_{\Ub_0}} S^*$ in the strong operator topology.
\end{enumerate}
Note that (P$_1$) and (P$_2$) directly follow from Corollary \ref{cor:GWP_rational}. As for (P$_3$), we remark that the strong operator convergence $T_{\Ub_{0,n}} \to T_{\Ub_0}$ follows from \eqref{eq:Urat_conv} and  by dominated convergence. By self-adjointness of $T_{\Ub_0,n}$ and $T_{\Ub_0}$ together with $T_{\Ub_{0,n}} \to T_{\Ub_0}$ strongly and $t_n \to t$, we easily deduce (P$_3$).

We observe the following convergence properties.
\begin{lem} \label{lem:GWP_minor}
Let $\Ub_0 \in H^{1/2}(\T; \Gr_k(\C^d))$. Suppose that the sequence $\Ub_{0,n} \in \mathcal{R}at(\T; \Gr_k(\C^d))$ with $n \in \N$ satisfies $\Ub_{0,n} \to \Ub_0$ in $H^{1/2}$ and let $t_n \to t$. Then 
$$
\mbox{$\Us_n(t_n) \Pi \Ub_{0,n} \to \Us(t) \Pi \Ub_0$ in $L^2$} \quad \mbox{and} \quad \mbox{$\Us_n(t_n) \Pi \Ub_{0,n} \weakto \Us(t) \Pi \Ub_0$ in $H^{1/2}$}
$$
In particular, the limit $\Us(t) \Pi \Ub_0$ is independent of the chosen sequence of rational initial data $\Ub_{0,n} \in \mathcal{R}at(\T; \Gr_k(\C^d))$. 
\end{lem}

\begin{proof}
First, we notice that
\be \label{ineq:conv_Un}
\| \Us_n(t_n) \Pi \Ub_{0,n} - \Us_n(t_n) \Pi \Ub_0 \|_{L^2} \leq \| \Pi ( \Ub_{0,n} - \Ub_0 )\|_{L^2} \to 0
\ee
using that $\| \Us_n(t_n) \| \leq 1$ for all $n \in \N$. From the strong operator convergence (P$_3$) and Lemma \ref{lem:Us_Sigma_conv}, we deduce that $\Us_n(t_n) \Pi \Ub_0 \weakto \Us(t) \Pi \Ub_0$ in $L^2_+$.  Together with \eqref{ineq:conv_Un} this implies
$$
\Us_n(t_n) \Pi \Ub_{0,n} \weakto \Us(t) \Pi \Ub_0 \quad \mbox{in $L^2$}.
$$
Evidently, the limit is independent of the chosen sequence $\Ub_{0,n}$ of rational data. Next, since $\Ub_{0,n}$ with $n \in \N$ forms a bounded sequence in $H^{1/2}$ and by conservation laws for the corresponding solutions $\Ub_n \in C(\R; H^\infty)$ in Corollary \ref{cor:GWP_rational}, we use (P$_2)$ to find 
$$
\sup_{n \in \N} \| \Us_n(t_n) \Pi \Ub_{0,n} \|_{H^{1/2}} \leq \sup_{n \in \N} \| \Ub_{0,n} \|_{H^{1/2}} < \infty.
$$
By  Rellich compactness $H^{1/2}(\T) \subset L^2(\T)$, we infer that $\Us_n(t_n) \Pi \Ub_{0,n} \to \Us(t) \Pi \Ub_0$ in $L^2$ and $\Us_n(t_n) \Pi \Ub_{0,n} \weakto \Us(t) \Pi \Ub_0$ in $H^{1/2}$. 
\end{proof}

As a next step, we introduce the map 
\be \label{def:phi}
\Phi : \R \times \mathcal{R}at(\T; \Gr_k(\C^d)) \to  \mathcal{R}at(\T; \Gr_k(\C^d))
\ee
by setting
$$
(t, \Ub_0) \mapsto \Phi_t(\Ub_0) := \Us(t) \Pi \Ub_0 + (\Us(t) \Pi \Ub_0)^* - \Mean(\Ub_0).
$$
For all $\Ub_0 \in \mathcal{R}at(\T; \Gr_k(\C^d))$, we note that $\Phi_t(\Ub_0) \in C(\R; H^\infty)$ yields the unique global-in-time solution of \eqref{eq:HWM} with rational initial datum $\Phi_0(\Ub_0)=\Ub_0$. 

\begin{lem} \label{lem:extension}
The map $\Phi$ defined in \eqref{def:phi} admits a unique continuous extension 
$$
\Phi : \R \times H^{1/2}(\T; \Gr_k(\C^d)) \to  H^{1/2}(\T; \Gr_k(\C^d)).
$$
Moreover, for any $\Ub_0 \in H^{1/2}(\T; \Gr_k(\C^d))$, we have that $\Ub(t) = \Phi_t(\Ub_0)$ is a weak solution of \eqref{eq:HWM} with initial datum $\Ub(0) = \Ub_0$ and it holds
$$
\Ub(\cdot) \in C(\R; L^2) \cap C(\R; H^{1/2}_{\mathrm{w}}),
$$
where $H^{1/2}_{\mathrm{w}}$ denotes $H^{1/2}$ equipped with the weak topology. In addition,
$$
\Mean(\Ub(t)) = \Mean(\Ub_0) \quad \mbox{and} \quad E[\Ub(t)] \leq E[\Ub_0] \quad \mbox{for all $t \in \R$}.
$$
\end{lem}

\begin{proof}
We take a sequence $\Ub_{0,n} \in \mathcal{R}at(\T;\Gr_k(\C^d))$ such that $\Ub_{0,n} \to \Ub_0$ in $H^{1/2}$ and assume that $t_n \to t$. By Lemma \ref{lem:GWP_minor} and the definition of $\Phi_t$, we conclude that
\be \label{eq:Phi_conv}
\mbox{$\Phi_{t_n}(\Ub_{0,n}) \to \Phi_t(\Ub_0)$ in $L^2$} \quad \mbox{$\Phi_{t_n}(\Ub_{0,n}) \weakto \Phi_t(\Ub_0)$ in $H^{1/2}$}.
\ee 
 This shows that the map $\Phi$ initially defined for rational initial data extends uniquely to initial data in $H^{1/2}(\T; \Gr_k(\C^d))$ so that the map
 $$
 \Phi : \R \times H^{1/2}(\T; \Gr_k(\C^d)) \to \R \times H^{1/2}_{\mathrm{w}}(\T; \C^{d \times d})
 $$
 is continuous, where $H^{1/2}_{\mathrm{w}}$ denotes $H^{1/2}$ equipped with the weak topology.

Let us write $\Ub(t) = \Phi_t(\Ub_0)$. We claim that
$$
\Ub(t, \theta) \in \Gr_k(\C^d) \quad \mbox{for all $t \in \R$ and a.e.~$\theta \in \T$}.
$$
Indeed, let us fix $t \in \R$. In view of \eqref{eq:Phi_conv}, we conclude that $\Ub_n(t, \theta) \to \Ub(t, \theta)$ in $\C^{d \times d}$ for almost every $\theta \in \T$ with the smooth rational solutions $\Ub_n \in C(\R; H^\infty)$ of \eqref{eq:HWM}. This pointwise convergence implies that the limit $\Ub(t)$ belongs to $H^{1/2}(\T; \Gr_k(\C^d))$ for all $t \in \R$.

From the previous discussion we see that 
$$
\Ub(t) = \Phi_t(\Ub_0) \in C(\R; L^2) \cap C(\R; H^{1/2}_{\mathrm{w}}).
$$
By passing to the limit $n \to \infty$ for the smooth rational solutions $\Ub_n(t) \in C(\R; H^\infty)$ and the fact that $\pt_t \Ub_n$ is uniformly bounded in $C(\R; H^{-1/2})$, we readily deduce that passing to the limit $n \to \infty$ yields that $\Ub(t) \in C(\R;L^2) \cap C(\R; H^{1/2}_{\mathrm{w}}$ is a weak solution of (HWM) with initial datum $\Ub(0)=\Ub_0$.

Finally, from the conservation laws $\Mean(\Ub_n(t)) = \Mean(\Ub_{0,n})$ and $E[\Ub_n(t)] = E[\Ub_{0,n}]$ for the smooth rational solutions $\Ub_n \in C(\R; H^\infty)$ together with \eqref{eq:Phi_conv} and the fact $\Ub_{0,n} \to \Ub_0$ in $H^{1/2}$, we easily deduce that
$$
\Mean[\Ub(t)] = \Mean[\Ub_0] \quad \mbox{and} \quad E[\Ub(t)] \leq E[\Ub_0] \quad \mbox{for all $t \in \R$},
$$
which finishes the proof.
\end{proof}

\subsection{Lax evolution and energy conservation}
At this point, it is conceivable that the solution $\Ub(t) = \Phi_t(\Ub_0)$ obtained above may exhibit a loss of energy, i.e., 
$$
E[\Ub(t)] < E[\Ub_0] \quad \mbox{for some $t \in \R$},
$$
which is tantamount to the failure of strong continuity of the map $t \mapsto \Phi_t(\Ub_0)$ in $H^{1/2}$. To rule out this scenario, and thus proving energy conservation as a by-product, we will derive that 
$$
\mbox{$\Us(t) : L^2_+ \to L^2_+$ is {\em unitary} for all $t \in \R$}
$$
and energy conservation will follow by a corresponding Lax evolution by $\Us(t)$ for the trace-class operator $K_{\Ub(t)} = \id - T_{\Ub(t)}^2$ with $E[\Ub(t)] = \Tr(K_{\Ub(t)})$. 

We start with the following key observation concerning the transport of eigenfunctions of $T_{\Ub_0}$ in terms of the map $\Us(t)$.

\begin{lem} \label{lem:eigenfunctions}
Let $t \in \R$ be given. Suppose $\Fb \in L^2_+ \setminus \{ 0 \}$ is an eigenfunction for $T_{\Ub_0}$ with eigenvalue $\mu \in \sigma(T_{\Ub_0})$. Then $\Fb(t) := \Us(t) \Fb$ satisfies
$$
T_{\Ub(t)} \Fb(t) = \mu \Fb(t) \quad \mbox{and} \quad \Fb(t) \neq 0.
$$
\end{lem}

\begin{proof}
We divide the proof into the following steps.

\medskip
\textbf{Step 1.} As before, let $\Ub_{0,n} \in \mathcal{R}at(\T; \Gr_k(\C^d))$ be a sequence of rational data with $\Ub_{0,n} \to \Ub_0$ in $H^{1/2}$. Likewise, we denote by $\Ub_n(t)$ the corresponding solutions $C(\R; H^\infty)$ of (HWM) with $\Ub_n(0) = \Ub_{0,n}$ and we let $\Us_n(t) : L^2_+ \to L^2_+$ denote the corresponding sequence of unitary maps. From the Lax evolution for rational solutions we infer that
\be \label{eq:Lax_eigen}
T_{\Ub_n(t)} \Us_n(t) \Fb = \Us_n(t) T_{\Ub_0} \Fb = \mu \Us_n(t) \Fb
\ee
where $\Fb \in L^2_+ \setminus \{ 0 \}$ solves $T_{\Ub_0} \Fb = \mu \Fb$. Note that $\Us_n(t) \Fb \weakto \Us(t) \Fb$ in $L^2$ analogous to the proof of Lemma \ref{lem:GWP_minor}. Moreover, from $\Ub_n(t) \to \Ub(t)$ in $L^2$ we easily deduce that $T_{\Ub_n(t)} \to T_{\Ub(t)}$ strongly as operators. By this fact and the self-adjointness of $T_{\Ub_n(t)}$ and $T_{\Ub(t)}$, it is elementary to check that $T_{\Ub_n(t)} \Us_n(t) \Fb \weakto T_{\Ub(t)} \Us(t) \Fb$. By passing to the limit $n \to \infty$ in \eqref{eq:Lax_eigen}, we obtain that $\Fb(t) := \Us(t) \Fb$ satisfies
$$
T_{\Ub(t)} \Fb(t) = \mu \Fb(t) .
$$

\medskip
\textbf{Step 2.} We now claim that 
$$
\Fb(t) \neq 0 .
$$
Recalling Lemma \ref{lem:UsSigma}, we see that $\Fb(t) = \Us(t) \Fb \neq 0$ provided that $\Mean(\Fb) \neq 0$. Hence it remains to discuss the case when $\Mean(\Fb) = 0$ for the rest of the proof.

Suppose now that $\Fb \neq 0$ with $\Mean(\Fb)=\widehat{\Fb}_0=0$. Since $\Fb \neq 0$, we can define 
$$
n := \min \{ k \geq 1 : \Mean((S^*)^k \Fb) = \widehat{\Fb}_k \neq 0 \}.
$$
Next, we consider the eigenspace
$$
\Es_{\mu} = \ker ( T_{\Ub_0} - \mu \id).
$$
As an eigenspace for the Toeplitz operator $T_{\Ub_0}$, we recall that $\Es_\mu$ must be {\em nearly $S^*$-invariant}, i.e., 
$$
\Gb \in \Es_\mu \quad \mbox{and} \quad \Mean(\Gb) = 0 \quad \Rightarrow \quad S^* \Gb \in \Es_\mu,
$$
which is a direct consequence of the commutator identity in Lemma \ref{lem:commutator}. Therefore, 
$$
(S^*)^j \Fb \in \Es_{\mu} \quad \mbox{for $j=0, \ldots, n$}.
$$ 
Now we recall the intertwining identity from Lemma \ref{lem:UsSigma}, which gives us
\be \label{eq:intertwine}
S^* \Us(t) = \Us(t) \eu^{-it T_{\Ub_0}} S^*.
\ee
By iterating this identity $n$ times and applying it to the vector $\Fb$, we obtain
$$
(S^*)^n \Us(t) \Fb = \Us(t) (\eu^{-it T_{\Ub_0}} S^*)^n \Fb.
$$
Since $\eu^{-it T_{\Ub_0}} \Gb = \eu^{-it \mu} \Gb$ for $\Gb \in \Es_\mu$, we find $(\eu^{-it T_{\Ub_0}} S^*)^n \Fb= \eu^{-in \mu t} (S^*)^n \Fb$, whence it follows that
$$
(S^*)^n \Us(t) \Fb= e^{-in \mu t} \Us(t) (S^*)^n \Fb \neq 0,
$$
using that $\Mean((S^*)^n \Fb) \neq 0$ and the fact that $\Us(t) \Gb \neq 0$ if $\Mean(\Gb) \neq 0$ by Lemma \ref{lem:UsSigma}. Thus we deduce that $\Fb(t) = \Us(t) \Fb \neq 0$ must hold.
\end{proof}

We next establish $\Us(t) :L^2_+ \to L^2_+$ has trivial kernel for all $t \in \R$, which proves that it is a unitary map thanks to Theorem \ref{thm:stony}.

\begin{lem} \label{lem:Us_kernel}
It holds that
$$
\ker \, \Us(t) = \{0 \} \quad \mbox{for all $t \in \R$}.
$$
\end{lem}

\begin{proof}
Let $t \in \R$ be given. For $\mu \in \sigma(T_{\Ub_0})$, we consider the eigenspaces
$$
\Es_\mu(0) = \ker (T_{\Ub_0} - \mu \id) \quad \mbox{and} \quad \Es_\mu(t)= \ker (T_{\Ub(t)} - \mu \id).
$$
By Lemma \ref{lem:eigenfunctions}, we infer that
$$
\mbox{$\Us(t) : \Es_\mu(0) \to \Es_\mu(t)$ is injective.}
$$
Assume now that there exists some $\Fb \in L^2_+$ with $\Fb \neq 0$ such that $\Us(t) \Fb = 0$. From Section \ref{sec:spec} we recall that $\sigma(T_{\Ub_0}) = \sp(T_{\Ub_0})$ is at most countable and only point spectrum. By self-adjointness of $T_{\Ub_0}$, we can write
$$
\Fb = \sum_{\mu \in \sigma(T_{\Ub_0})} \Fb_\mu \quad \mbox{with $\Fb_\mu \in \Es_\mu(0)$}.
$$
Since $\Us(t)$ maps $\Es_\mu(0)$ into $\Es_\mu(t)$, we observe that
$$
\Us(t) \Fb = \sum_{\mu  \in \sigma(T_{\Ub_0})} \Us(t) \Fb_\mu \quad \mbox{with $\Us(t) \Fb_\mu \in \Es_\mu(t)$}.
$$
Because $\Es_\mu(t) \perp \Es_\nu(t)$ whenever $\mu \neq \nu$ by self-adjointness of $T_{\Ub(t)}$, it follows from $\Us(t) \Fb = 0$ that
$$
\Us(t) \Fb_\mu = 0 \quad \mbox{for all $\mu \in \sigma(T_{\Ub_0})$}.
$$
On the other hand, since $\Fb \neq 0$, there exists some eigenvalue $\mu_* \in \sigma(T_{\Ub_0})$ such that $\Fb_{\mu_*} \neq 0$. Since $\Us(t) : \Es_{\mu_*}(0) \to \Es_{\mu_*}(t)$ is injective, this contradicts that $\Us(t) \Fb_{\mu_*} = 0$.
\end{proof}

With the aid of Lemma \ref{lem:Us_kernel}, we are ready to establish the following central result.

\begin{lem}[Lax evolution and energy conservation] \label{lem:energy_conservation}
For all $t\in \R$, the map $\Us(t) : L^2_+ \to L^2_+$ is unitary and it holds that
$$
T_{\Ub(t)} = \Us(t) T_{\Ub_0} \Us(t)^*.
$$
In particular, the trace-class operator $K_{\Ub(t)} = \id - T_{\Ub(t)}^2$ satisfies 
$$
K_{\Ub(t)} = \Us(t) K_{\Ub_0} \Us(t)^*,
$$
and we have energy conservation
$$
E[\Ub(t)] = \Tr(K_{\Ub(t)}) = \Tr(K_{\Ub_0}) = E[\Ub_0] \quad \mbox{for all $t \in \R$}.
$$
\end{lem}

\begin{proof}

That $\Us(t) : L^2_+ \to L^2_+$ is unitary for all $t\in \R$ follows from Theorem \ref{thm:stony} together with Lemma \ref{lem:Us_kernel}. 

To prove the unitary equivalence of $T_{\Ub(t)}$ and $T_{\Ub_0}$, let $\Ub_{0,n} \in \mathcal{R}at(\R; \Gr_k(\C^d))$ be a sequence of rational initial data such that $\Ub_{0,n} \to \Ub_0$ in $H^{1/2}$. As usual, we denote by $\Ub_n(t) = \Phi_t(\Ub_{0,n})$ the corresponding smooth global-in-time solutions. By the Lax evolution, it holds that
$$
T_{\Ub_n(t)} \Us_n(t) = \Us_n(t) T_{\Ub_{0,n}}
$$
with the corresponding unitary maps $\Us_n(t) : L^2_+ \to L^2_+$ for $n \in \N$. Since $\Us_n(t) \weakto \Us(t)$ and $T_{\Ub_n(t)} \to T_{\Ub(t)}$ strongly for any $t \in \R$, the self-adjointness of $T_{\Ub_n(t)}$ and $T_{\Ub(t)}$ allows us to pass to the limit $n \to \infty$ in the identity above to deduce that
$$
T_{\Ub(t)} \Us(t) = \Us(t) T_{\Ub_0}.
$$
Thanks to the unitarity of $\Us(t) : L^2_+ \to L^2_+$, we infer that $T_{\Ub(t)} = \Us(t) T_{\Ub_0} \Us(t)^*$.

The unitary equivalence $K_{\Ub(t)} = \Us(t) K_{\Ub_0} \Us(t)^*$ is now a trivial consequence of the identity $K_{\Ub(t)} = \id - T_{\Ub(t)}^2$. Finally, we recall from Lemma \ref{lem:key_identity} that $E[\Ub(t)] = \Tr(K_{\Ub(t)})$. Together with the fact that $\Tr(U K U^*) = \Tr(K)$ for unitary maps $U$ and trace-class operators $K$ this completes the proof. 
\end{proof}

\subsection{Proof of Theorem \ref{thm:gwp}}
We now ready to prove global well-posedness for (HWM) with initial data in $H^{1/2}$ as stated in Theorem \ref{thm:gwp}. In fact, the proof follows for directly from the preceding discussion. For the reader's convenience, we detail the arguments as follows.

First, we show the continuity of the map 
$$
\Phi : \R \times H^{1/2}(\T; \Gr_k(\C^d)) \to H^{1/2}(\T; \Gr_k(\C^d)),
$$ 
which was obtained as a unique extension from rational initial data by Lemma \ref{lem:GWP_minor}. Recall that $\Phi$ is known to be weakly continuous. Thus it remains to show that 
$$
\| \Phi_{t_n}(\Ub_{0,n}) \|_{H^{1/2}} \to \| \Phi_t(\Ub_0) \|_{H^{1/2}}
$$
whenever $(t_n, \Ub_{0,n}) \to (t, \Ub_0)$ in $\R \times H^{1/2}(\T;\Gr_k(\C^d))$. Indeed, using the fact that $\| \Fb \|_{H^{1/2}}^2 = |\Mean(\Fb)|^2 + 2 E[\Fb]$ for $\Fb \in H^{1/2}$ together with the conservation of mean and energy conservation by Lemma \ref{lem:GWP_minor} and \ref{lem:energy_conservation}, we directly deduce that
\begin{align*}
\| \Phi_{t_n}(\Ub_{0,n}) \|_{H^{1/2}}^2 = \| \Ub_{0,n} \|_{H^{1/2}}^2 \to \| \Ub_0 \|_{H^{1/2}}^2 = \| \Phi_t(\Ub_0) \|_{H^{1/2}}^2
\end{align*}
provided that $\Ub_{0,n} \to \Ub_0$ in $H^{1/2}$.

\medskip
(i)--(ii).  By the continuity of $\Phi$, we deduce that $\Ub \in C(\R; H^{1/2}(\T; \Gr_k(\C^d)))$ holds and from Lemma \ref{lem:GWP_minor} we recall that $\Ub(t) = \Phi_t(\Ub_0)$ is a weak solution of (HWM) with $\Ub(0) = \Ub_0$.  Also, we have energy conservation $E[\Ub(t)] = E[\Ub_0]$ for all $t \in \R$ thanks to Lemma \ref{lem:energy_conservation}.

\medskip
(iii).  Suppose that $\Ub_{0,n} \in H^{1/2}(\T; \Gr_k(\C^d))$ is a sequence with $\Ub_{0,n} \to \Ub_0$ in $H^{1/2}$. For any compact interval $I \subset \R$, we claim that
$$
\sup_{t \in I} \| \Phi_{t}(\Ub_{0,n}) - \Phi_t(\Ub_0) \|_{H^{1/2}} \to 0 \quad \mbox{as $n \to \infty$}.
$$
We argue by contradiction and assume that there exist a sequence $t_n \in I$ and some $\eps_0 > 0$ such that $\| \Phi_{t_n}(\Ub_{0,n}) - \Phi_t(\Ub_0) \|_{H^{1/2}} \geq \eps_0$ for all $n \in \N$. By compactness of $I$, we can assume that $t_n \to t$ for some $t \in I$ by passing to a subsequence if necessary. But since $\Phi_{t_n}(\Ub_{0,n}) \to \Phi_{t}(\Ub_0)$ in $H^{1/2}$ by the continuity of $\Phi$, we arrive at a contradiction.

\medskip
(iv). The group property 
$$
\Phi_{t+s}(\Ub_0) = \Phi_t (\Phi_s(\Ub_0)) \quad \mbox{for $t,s \in \R$}
$$
holds for rational initial data $\Ub_0 \in \mathcal{R}at(\T; \Gr_k(\C^d))$ by uniqueness of these solutions in the class $C(\R; H^{1/2})$. By density of rational data and continuity of $\Phi$, we extend the group property above to any $\Ub_0 \in H^{1/2}(\T; \Gr_k(\C^d))$.

\medskip
(v). We recall from Corollary \ref{cor:GWP_rational} that $\Phi_t(\Ub_0) \in \mathcal{R}at(\T; \Gr_k(\C^d))$ holds for all $t \in \R$ provided that $\Ub_0 \in \mathcal{R}at(\T; \Gr_k(\C^d))$. 

\medskip
Finally, let us suppose that 
$$
\widetilde{\Phi} : \R \times H^{1/2}(\T; \Gr_k(\C^d)) \to H^{1/2}(\T; \Gr_k(\C^d))
$$
is a continuous map such that properties (i) and (v) hold. For $\Ub_0 \in \mathcal{R}at(\T; \Gr_k(\C^d))$, this is seen to imply that $\widetilde{\Ub}(t) = \widetilde{\Phi}_t(\Ub_0)$ is a rational solution of (HWM) in $C(\R; H^\infty)$. But by Corollary \ref{cor:GWP_rational} it follows that
$$
\Pi \widetilde{\Ub}(t,z) = \Mean \left ( ( \id - z \eu^{-i t T_{\Ub_0}} S^*)^{-1} \Pi \Ub_0 \right )
$$ 
is given by the explicit formula. Thus $\Phi$ and $\widetilde{\Phi}$ coincide on rational initial data, and by the uniqueness of the extension to $H^{1/2}$, we conclude that $\Phi$ and $\widetilde{\Phi}$ are identical. This proves uniqueness of the flow map $\Phi$.

The proof of Theorem \ref{thm:gwp} is now complete. \hfill $\Box$

\section{Almost Periodicity}

\label{sec:ap}

\subsection{Preliminaries} 
For the reader's convenience, we first collect some fundamental facts about almost periodic functions valued in Banach spaces. 

Let $X$ be a Banach space and denote by $\BC(\R; X)$ the space of bounded and continuous functions from $\R$ to $X$, which is a Banach space endowed with the norm 
$\| f \|_\BC = \sup_{t \in \R} \| f(t) \|_X$. We recall that $f \in \BC(\R;X)$ is said to be {\em almost periodic} (in the sense of Bohr) if for every $\eps > 0$, there exists $L > 0$ such that every interval $I \subset \R$ of length $L$ contains some $\tau \in I$ with
$$
\sup_{t \in \R} \| f(t + \tau) - f(t) \|_X \leq \eps.
$$
Recall that {\em Bochner's criterion} states that $f \in \BC(\R;X)$ is almost periodic if and only if the set of all its translates 
$$
\mathrm{Trans}(f) = \{ f(\cdot + a) : a \in \R \}
$$
is relatively compact in the Banach space $\BC(\R;X)$. In particular, we see that the space of almost periodic functions
$$
\AP(\R;X) := \{ f \in \BC(\R;X) : \mbox{$f : \R \to X$ is almost periodic} \}
$$
is a Banach space in its own right, equipped with the $\sup$-norm.

\subsection{Almost periodicity for (HWM)}
The goal of this subsection is to prove 
$$
\Phi_t(\Ub_0) \in \AP(\R; H^{1/2}) \quad \mbox{for all $\Ub_0 \in H^{1/2}(\T; \Gr_k(\C^d))$}
$$
for the solutions of (HWM) as provided by Theorem \ref{thm:gwp}. Again, we will make strong use of the explicit formula for (HWM) as follows.

As usual, we let $d \geq 2$ and $1 \leq k \leq d-1$ be fixed integers and we suppose that $\Ub_0 \in H^{1/2}(\T; \Gr_k(\C^d))$ is an initial datum for (HWM). For notational convenience, we denote
$$
\Omega(t) := \eu^{-i t T_{\Ub_0}} : L^2_+ \to L^2_+ \quad \mbox{for $t \in \R$}
$$
with the Hardy space $L^2_+ = L^2_+(\T;\C^{d \times d})$. Using that the spectrum $\sigma(T_{\Ub_0})$ is at most countable, we establish with the following compactness property for the strongly continuous one-parameter unitary group $\{ \Om(t) \}_{t \in \R}$.

\begin{lem} \label{lem:ap_compact}
For every sequence $(a_n)$ in $\R$, there is a subsequence  $(a'_n)$ such that
$$
\mbox{$\Om(a'_n) \to \Om_\infty$ strongly},
$$
where $\Om_\infty$ is a unitary map in $L^2_+$. 
\end{lem}

\begin{remark*}
For notational convenience, we use $(a'_n)$ to denote subsequences of $(a_n)$, since the real numbers $a_n$ appear in $\Om(a_n) = \eu^{-i a_n T_{\Ub_0}}$ and so introducing further subindices in the exponent would lead to clumsy notation.
\end{remark*}

\begin{proof}
From Lemma \ref{lem:Toeplitz_spectrum}, we recall the spectral representation
$$
T_{\Ub_0} = \sum_{\mu \in \sigma(T_{\Ub_0})} \mu P_\mu,
$$
where $P_\mu$ denotes the orthogonal projector onto the eigenspace $\ker(T_{\Ub_0} - \mu \id)$. As a consequence from standard spectral calculus, we have
$$
\Om(a_n) = \eu^{-i a_n T_{\Ub_0}} = \sum_{\mu \in \sigma(T_{\Ub_0})} \eu^{-i a_n \mu} P_\mu.
$$
Since $\sigma(T_{\Ub_0})$ is at most countable, there exist a subsequence denoted by $(a'_n)$ and some at most countable set $\{ \omega(\mu) \}_{\mu \in \sigma(T_{\Ub_0})} \subset \Ss^1$ such that
$$
\eu^{-i a'_n \mu} \to \omega(\mu) \quad \mbox{as $n \to \infty$} \quad \mbox{for any $\mu \in \sigma(T_{\Ub_0})$}.
$$
Consequently, the unitary operators $\Om(a'_n)=\eu^{-i a'_n T_{\Ub_0}}$  converge strongly to 
$$
\Om_\infty := \sum_{\mu \in \sigma(T_{\Ub_0})} \omega(\mu) P_\mu,
$$
which is evidently a unitary operator in $L^2_+$. 
\end{proof}

As before, let us denote by
$$
\Ub(t) = \Phi_t(\Ub_0) = (\Us(t) \Pi \Ub_0) + (\Us(t) \Pi \Ub_0)^* - \Mean(\Ub_0)
$$
 the solution given by the flow map for (HWM) as provided by Theorem \ref{thm:gwp}. We deduce  the following convergence lemma.

\begin{lem} \label{lem:ap_convergence}
Let $(t_n)$ be a sequence in $\R$ such that $\Om(t_n) \to \Om_\infty$ strongly with some unitary map $\Om_\infty :L^2_+ \to L^2_+$. Then $\Ub(t_n) \to \Ub_\infty$ in $H^{1/2}$, where the limit $\Ub_\infty$ belongs to $H^{1/2}(\T; \Gr_k(\C^d))$ and it satisfies 
$$
E[\Ub_\infty] = E[\Ub_0], \quad \Pi \Ub_\infty(z) = \Mean ( (\id - z  \Om_\infty S^*)^{-1} \Pi \Ub_0) \quad \mbox{for $z \in \D$}.
$$

\end{lem}

\begin{proof}
We adapt the arguments used in Section \ref{sec:gwp} to show this result.

\medskip
\textbf{Step 1.} Since $\Om(t_n) \to \Om_\infty$ strongly and by mirroring the proof of Lemmas \ref{lem:GWP_minor} and \ref{lem:extension}, we deduce that
$$
\mbox{$\Ub(t_n) \to \Ub_\infty$ in $L^2$ and $\Ub(t_n) \weakto \Ub_\infty$ in $H^{1/2}$}.
$$
Here the limit $\Ub_\infty$ belongs to $H^{1/2}(\T; \Gr_k(\C^d))$ and satisfies the explicit formula
 $$
 \Pi \Ub_\infty(z) = \Mean \left ( ( \id - z\Om_\infty S^*)^{-1} \Pi \Ub_0 \right ) \quad \mbox{for $z \in \D$}.
 $$
Furthermore, from Lemma \ref{lem:UsSigma} we obtain the intertwining relation
\be \label{eq:ap1}
S^* \Us_\infty = \Us_\infty \Om_\infty S^*,
\ee
where $\Us_\infty : L^2_+ \to L^2_+$ is the contraction given by
$$
(\Us_\infty \Fb)(z) = \Mean ((\id-z \Om_\infty S^*)^{-1} \Fb).
$$
for any $\Fb \in L^2_+$ and $z \in \D$.

\medskip
\textbf{Step 2.} We claim that
$$
\mbox{$\Us_\infty : L^2_+ \to L^2_+$ is unitary.}
$$
In view of Theorem \ref{thm:stony}, this amounts to showing that $\ker \, \Us_\infty = \{ 0 \}$ holds. To see this, let us define the Toeplitz operator $T_{\Ub_\infty}$ with symbol $\Ub_\infty \in H^{1/2}(\T; \Gr_k(\C^d))$. Since $\Us(t_n) \to \Us_\infty$ weakly and $T_{\Ub(t_n)} \to T_{\Ub_\infty}$ strongly as operators, we can pass to the limit $n \to \infty$ in the Lax evolution $T_{\Ub(t_n)} \Us(t_n) = \Us(t_n) T_{\Ub_0}$. This yields 
\be \label{eq:ap2}
T_{\Ub_\infty} \Us_\infty = \Us_\infty T_{\Ub_0}.
\ee
Now, by exactly following the arguments in the proof of Lemma \ref{lem:eigenfunctions} and using the intertwining relation \eqref{eq:ap1} above, we infer that
$$
T_{\Ub_0} \Fb = \mu \Fb \quad \mbox{and} \quad \Fb \neq 0 \quad \Rightarrow \quad T_{\Ub_\infty} \Us_\infty \Fb = \mu \Us_\infty \Fb \quad \mbox{and} \quad \Us_\infty \Fb \neq 0.
$$ 
Thus the map $\Us_\infty : \ker (T_{\Ub_0} - \mu \id) \to \ker (T_{\Ub_\infty} - \mu \id)$ is injective and, by a direct adaptation of the proof of Lemma \ref{lem:Us_kernel}, we deduce that $\ker \, \Us_\infty = \{ 0 \}$ is trivial. This shows that $\Us_\infty$ is unitary by Theorem \ref{thm:stony}.

\medskip
\textbf{Step 3.} Since $\Us_\infty$ is unitary, we see from \eqref{eq:ap2} that $T_{\Ub_\infty} = \Us_\infty T_{\Ub_0} \Us_\infty^*$. Therefore $K_{\Ub_\infty} = \Us_\infty K_{\Ub_0} \Us_\infty^*$ with the trace-class operator $K_{\Ub_\infty} = \id - T_{\Ub_\infty}^2$ and we deduce that
$$
E[\Ub_\infty] = \Tr (K_{\Ub_\infty}) =  \Tr(\Us_\infty K_{\Ub_0} \Us_\infty^*) =  \Tr(K_{\Ub_0}) = E[\Ub_0].
$$
In view of $E[\Ub(t_n)] = E[\Ub_0]$ this implies that we must have strong convergence  $\Ub(t_n) \to \Ub_\infty$ in $H^{1/2}$.
\end{proof}

\subsection*{Proof of Theorem \ref{thm:ap}}
Let $\Ub_0 \in H^{1/2}(\T; \Gr_k(\C^d))$ be an initial datum for (HWM). As usual, we denote by $\Ub(t) = \Phi_t(\Ub_0) \in C(\T;H^{1/2})$ the solution provided by Theorem \ref{thm:gwp}.
 
To show that $\Ub \in \AP(\R; H^{1/2})$ holds, we use Bochner's criterion for almost periodicity. Thus we have to prove that, for every sequence $(a_n)$ in $\R$, there exists a subsequence $(a'_n)$ such that the sequence of bounded and continuous functions 
$$
\R \to H^{1/2}(\T; \C^{d \times d}), \quad t \mapsto \Ub(t+a'_n) 
$$
is uniformly convergent on $\R$ as $n \to \infty$. 

Indeed, given a sequence $(a_n)$ in $\R$, we choose the subsequence $(a'_n)$ provided by Lemma \ref{lem:ap_compact} ensuring that
$$
\Om(a'_n) \to \Om_\infty 
$$
strongly as operator with some unitary map $\Om_\infty$ in $L^2_+$. For $t \in \R$, we define 
$$
\Ub_\infty(t) := \Pi \Ub_\infty(t) + (\Pi \Ub_\infty(t))^* - \Mean(\Pi \Ub_0) 
$$
where we set
$$
\Pi \Ub_\infty(t,z) := \Mean \left ( \id- z \Om(t) \Om_\infty S^*)^{-1} \Pi \Ub_0 \right ) \quad \mbox{with $\Om(t)= \eu^{-it T_{\Ub_0}}$}.
$$
 Evidently, we have that $\Om(t) \Om(a'_n) \to \Om(t) \Om_\infty$ strongly for any $t \in \R$. Thus by Lemma \ref{lem:ap_convergence} applied to the sequence $(t_n) = (t +a'_n)$, we deduce 
$$
\mbox{$\Ub(t+a'_n) \to \Ub_\infty(t)$ in $H^{1/2}$ for any $t \in \R$}.
$$ 
 Now, we claim that we have in fact uniform convergence on $\R$, i.e.,
\be \label{ineq:ap_sup}
\sup_{t \in \R} \| \Ub(t+a'_n) - \Ub_\infty(t) \|_{H^{1/2}} \to 0 \quad \mbox{as} \quad n \to \infty.
\ee
Note that the sup is finite, since $\Ub \in BC(\R;H^{1/2})$ and $\sup_{t \in \R} \| \Ub_\infty(t) \|_{H^{1/2}}  < \infty$ thanks to $E[\Ub_\infty(t)] = E[\Ub_0]$ for all $t \in \R$ by Lemma \ref{lem:ap_convergence}. Thus we  can find a sequence $(t_n)$ in $\R$ such that
\be \label{ineq:ap_sup2}
\sup_{t \in \R} \| \Ub(t+a'_n) - \Ub_\infty(t) \|_{H^{1/2}} \leq \| \Ub(t_n + a'_n) - \Ub_\infty(t_n) \|_{H^{1/2}} + 2^{-n}.
\ee
 Applying Lemma \ref{lem:ap_compact} again, we can extract a subsequence $(t'_n)$ such that $\Om(t'_n) \to \widetilde{\Om}_\infty$ strongly with some unitary operator $\widetilde{\Om}_\infty$. By the group property of $\{ \Om(t) \}_{t \in \R}$, we immediately find that
 $$
\mbox{$\Om(t'_n + a'_n) = \Om(t'_n) \Om(a'_n) \to \widetilde{\Om}_\infty \Om_\infty$ strongly.}
$$
By Lemma \ref{lem:ap_convergence}, this implies that 
$$
\mbox{$\Ub(t'_n + a'_n) \to \Ub_\infty$ in $H^{1/2}$},
$$
where the limit $\Ub_\infty$ satisfies
\be \label{eq:U_infinity}
\Pi \Ub_\infty(z) = \Mean((\id- z \widetilde{\Om}_\infty \Om_\infty S^*)^{-1} \Pi \Ub_0) \quad \mbox{for $z \in \D$}.
\ee
On the other hand, from the explicit expression for $\Ub_\infty(t)$, we see that
$$
\Pi \Ub_\infty(t'_n,z) = \Mean ((\id-z \Om(t'_n) \Om_\infty S^*)^{-1} \Pi \Ub_0) \quad \mbox{for $z \in \D$}.
$$
Since $\Om(t'_n) \Om_\infty \to \widetilde{\Om}_\infty \Om_\infty$ strongly, we also deduce from Lemma \ref{lem:ap_convergence}  that 
$$
\mbox{$\Ub(t'_n) \to \Ub_\infty$ in $H^{1/2}$},
$$ 
where the limit $\Ub_\infty$ is again given by \eqref{eq:U_infinity}. Hence we have shown that
$$
 \| \Ub(t'_n+a'_n) - \Ub_\infty(t'_n) \|_{H^{1/2}} \to 0.
$$
By passing to the subsequence $(t'_n)$ in \eqref{ineq:ap_sup2} and taking the limit $n \to \infty$, we obtain the claimed uniform convergence \eqref{ineq:ap_sup}. This shows that $\Ub \in \AP(\R; H^{1/2})$ holds.

Finally, we remark that the Poincar\'e recurrence as stated in Theorem \ref{thm:ap} is an immediate consequence of Bohr's characterization of almost periodicity. Moreover, by the standard fact that almost periodic functions are uniformly continuous, we deduce that $\{ \Phi_t(\Ub_0) : t \in \R \}$ is a relatively compact subset in $H^{1/2}$.

The proof of Theorem \ref{thm:ap} is now complete. \hfill $\Box$

\begin{appendix}

\section{LWP in $H^s$ with $s > \frac 3 2$} 

\label{app:lwp}

\begin{lem}[LWP in $H^s$ with $s > \frac 3 2$] \label{lem:lwp}
Let $d \geq 2$ and $1 \leq k \leq d-1$ be integers and assume that $s > \frac{3}{2}$. Then, for every $\Ub_0 \in H^s(\T; \Gr_k(\C^d))$, there exists a unique solution $\Ub \in C(I; H^s)$ of \eqref{eq:HWM} with $\Ub(0)= \Ub_0$, and $I \subset \R$ denotes its corresponding maximal time interval of existence. 

Moreover, we have global-in-time existence, i.e., $I = \R$, provided that 
$$
\sup_{t \in I} \| \Ub(t) \|_{H^s} < \infty.
$$
\end{lem}

\begin{proof}
The proof of Lemma \ref{lem:lwp} follows from a direct adaptation of  \cite{GeLe-25}[Appendix D], where the analogous local well-posedness result for \eqref{eq:HWM} posed on $\R$ for $H^s$-data with $s > \frac 3 2$ is proved by a Kato-type iteration scheme. Note that the commutator estimates and fractional Leibniz rule used in \cite{GeLe-25} can be directly adapted to the setting on the torus; see, e.g., \cite{BeOhTe-25} for a systematic treatment of fractional Leibniz rule estimates on $\T$ from the corresponding estimates for functions on $\R$. 
\end{proof}

\begin{remark*}
A closer inspection of the proof shows that an a-priori Lipschitz bound 
$$
\sup_{t \in I} \| \pt_\theta \Ub(t) \|_{L^\infty} < \infty
$$
would be sufficient to deduce global existence in Lemma \ref{lem:lwp} above. However, as already pointed out in the introduction in Section 1, the Lax structure for \eqref{eq:HWM} fails to provide such an a-priori bound. 
\end{remark*}

Furthermore, by adapting the discussion in \cite{GeLe-25} for \eqref{eq:HWM} on $\R$, we deduce the following result for the operator-valued evolution equation
\be \label{eq:ode}
\frac{d}{dt} \Us(t) = B_{\Ub(t)} \Us(t) \quad \mbox{for $t \in I$}, \quad \Us(0) = \id
\ee
where $I \subset \R$ is a time interval with $0 \in I$ and given $\Ub \in C(I; H^s(\T; \Gr_k(\C^d)))$ with some $s > \frac{3}{2}$ together with the unbounded and (formally) skew-adjoint operator
$$
B_{\Ub} = \frac{i}{2} ( D  T_{\Ub} + T_{\Ub}  D ) - \frac{i}{2} T_{|D| \Ub}
$$
acting on $L^2_+ = L^2_+(\T; \C^{d \times d})$ with $D = -i \pt_\theta$.

\begin{lem}
Let $d \geq 2$ and $1 \leq k \leq d-1$ be integers and suppose that $\Ub \in C(I; H^s(\T; \Gr_k(\C^d)))$ solves \eqref{eq:HWM} with some $s > \frac{3}{2}$ on the time interval $I \subset \R$ with $0 \in I$. Then there exists a unique solution $\Us : I \to \Ls(L^2_+)$ of \eqref{eq:ode} with the following properties.
\begin{enumerate}
\item[(i)] The map $I \mapsto L^2_+$ with $t \mapsto \Us(t) \Fb$ is continuous for every $\Fb \in L^2_+$.
\item[(ii)] The equation $\pt_t \Us(t) = B_{\Ub(t)} \Us(t)$ holds in $H^{-1}_+(\T; \C^{d \times d})$ for any $t \in I$. 
\item[(iii)] $\Us(t) : L^2_+ \to L^2_+$ is a unitary map for any $t \in I$.
\item[(iv)] $\Us(t) \Fb \in H^1_+$ for any $\Fb \in H^1_+ = H^1 \cap L^2_+$.
\item[(v)] The operator $B_{\Ub(t)} : H^1_+ \subset L^2_+ \to L^2_+$ is essentially skew-adjoint.
\end{enumerate}
\end{lem}

\begin{remark*}
Note that (iii) means that there exists a unique skew-adjoint extension $B_{\Ub(t)}^* = -B_{\Ub(t)}$ with some domain $\dom(B_{\Ub(t)})$ and $H^1_+$ is dense in $\dom(B_{\Ub(t)})$  with respect to the graph norm of $B_\Ub(t)$.
\end{remark*}

\begin{proof}
The proof follows from a direct adaptation from the arguments in the proof of \cite{GeLe-25}[Lemma D.2]. See also the remark following \cite{GeLe-25}[Lemma D.2] about property (iii). 
\end{proof}

\section{Half-harmonic maps and solitary waves}

Let $d \geq 2$ and $1 \leq k \leq d-1$ be integers. We recall that critical points $\Qb \in H^{1/2}(\T; \Gr_k(\C^d))$ of the energy functional $E[\Ub]$ in \eqref{def:energy_HWM} are {\em half-harmonic maps} and they are (weak) solutions of 
\be \label{eq:Qb_hh}
[\Qb, |D| \Qb] = 0 .
\ee
Evidently, the half-harmonic maps   correspond to {\em stationary solutions} of \eqref{eq:HWM} with finite energy. As observed in \cite{LeSc-18}, we have indeed {\em traveling solitary waves} for \eqref{eq:HWM} with given velocity $v \in (-1,1)$. The corresponding profiles $\Qb_v \in H^{1/2}(\T; \Gr_k(\C^d))$ are (weak) solutions of \eqref{eq:Qb_hh} in the generalized form
\be
[\Qb, |D| \Qb] - \frac{i}{2} v\pt_\theta \Qb = 0.
\ee 

\subsubsection*{Special case $\Gr_1(\C^2) \cong \Ss^2$} In this case all half-harmonic maps are explicitly known in closed form.  Adapting this result to our matrix-valued formulation, we find that all half-harmonic maps $\Ub \in H^{1/2}(\T; \Gr_1(\C^2))$ are of the form
\be \label{eq:Qb_half}
\Qb(\theta) = U \left ( \begin{array}{cc} 0 & \ov{B(\eu^{i \theta})} \\ B(\eu^{i \theta}) & 0 \end{array} \right ) U^* 
\ee
where $U \in \C^{2 \times 2}$ is any constant unitary matrix, and $B: \D \to \D$ with
$$
B(z) = \eu^{i \phi} \prod_{k=1}^m B(z,a_k)
$$
is any finite {\em Blaschke product} of degree $m \in \N$ with zeros $a_1, \ldots, a_m \in \D$ and the Blaschke factors $B(z,a) =\frac{z-a}{1-\ov{a} z}$ for $a \neq 0$ and $B(z,0) = z$. From \eqref{eq:Qb_half} a direct calculation shows that {\em energy quantization} of half-harmonic maps with
$$
E[\Qb] = \pi \cdot m
$$
where $m$ is the degree of the Blaschke product in $\Qb$. The case $m=0$ yields the trivial case of constant half-harmonic maps, whereas $m=1$ corresponds to ground state half-harmonic maps (i.e.~nontrivial half-harmonic maps that globally minimize the energy). Moreover, by translating the results in \cite{LeSc-18} into matrix-valued notation, we recall that the traveling solitary profiles $\Qb_v$ with $v \in (-1,1)$ can be written as
\be
\Qb_v(\theta) = U \left ( \begin{array}{cc} v & (1-v^2)^{1/2} \ov{B(\eu^{i \theta})} \\ (1-v^2)^{1/2} B(\eu^{i \theta}) & -v \end{array} \right ) U^*,
\ee
with the energy $E[\Qb_v] = (1-v^2) \cdot \pi \cdot m$. Note that $E[\Qb_v] \to 0$ as $|v| \to 1$, and hence we can construct traveling solitary waves with arbitrarily small energy.

\subsubsection*{General case $\Gr_k(\C^d)$} For the general Grassmannians $\Gr_k(\C^d)$ as targets, we can easily build half-harmonic maps as follows. Let $d \geq 2$ and $1 \leq k \leq d-1$ be integers, where without loss of generality we can assume that $d \geq 3$ holds, since the remaining case $\Gr_1(\C^2)$ is already understood (see above). To construct half-harmonic maps from $\T$ into $\Gr_k(\C^d)$, we use that elements in $\Gr_1(\C^2)$ can be naturally embedded into $\Gr_k(\C^d)$ as follows. Let $\Qb \in H^{1/2}(\T; \Gr_1(\C^2))$ be a half-harmonic map, cf.~\eqref{eq:Qb_half}, and define 
$$
\widetilde{\Qb}(\theta) = \left ( \begin{array}{cc} \Qb(\theta) & 0 \\ 0 & J  \end{array} \right ),
$$ 
where $J$ is a constant matrix in $\Gr_{k-2}(\C^{d-2})$, e.g., take 
$$
J=\left ( \begin{array}{cc} \mathds{1}_p & 0 \\ 0 & -\mathds{1}_q  \end{array} \right )
$$ 
with $p,q \in \{0,\ldots, d-2 \}$ such that $\Tr(J) = p-q = d-2k$. We readily check that $\widetilde{\Qb} \in H^{1/2}(\T; \Gr_k(\C^d))$ solves \eqref{eq:Qb_half} and thus it is a half-harmonic map. Likewise, we can easily generalize this construction to obtain traveling solitary waves for \eqref{eq:HWM} with any target $\Gr_k(\C^d)$ by using the known case $\Gr_1(\C^2)$.

Finally, we remark that is is an interesting open question whether the construction above yields (up to symmetries) {\em all} half-harmonic maps from $\T$ to $\Gr_k(\C^d)$. We hope to address this question in future work.

\end{appendix}

\bibliographystyle{siam}
\bibliography{HWM_Bibliography}

\end{document}